\documentclass{IEEEtran}

\usepackage{cite}
\usepackage{graphicx,color}
\usepackage{amsmath,amssymb,amsfonts,amsthm}
\usepackage{amssymb}
\usepackage{mathrsfs}
\usepackage[algoruled,vlined]{algorithm2e}
\usepackage{subfigure}
\usepackage{dsfont}
\usepackage{booktabs}
\usepackage{array}
\usepackage{hyperref}

\usepackage{enumitem}
\usepackage{multirow}
\newtheorem{theorem}{Theorem}
\newtheorem{lemma}[theorem]{Lemma}
\newtheorem{definition}{Definition}
\newtheorem{corollary}[theorem]{Corollary}

\newtheorem{remark}[theorem]{Remark}
\newtheorem{assumption}{Assumption}

\newcommand{\setdef}[2]{\{#1 \; : \; #2\}}
\newcommand{\subscr}[2]{{#1}_{\textup{#2}}}

\newcommand\ajustspaceandequationnumber{%
   \vspace{-\belowdisplayskip}
   \vspace{-\abovedisplayskip}
   \addtocounter{equation}{-1}}

\newcommand{\real}{\mathbb{R}}

\newcommand{\trans}{\mathsf{T}} 
\newcommand{\conjtrans}{\mathsf{H}}

\newcommand{\mc}{\mathcal}

\newcommand{\F}{\bar{\mathsf{F}}}

\DeclareSymbolFont{bbold}{U}{bbold}{m}{n}
\DeclareSymbolFontAlphabet{\mathbbold}{bbold}

\newcommand\oprocendsymbol{\hbox{$\square$}}
\newcommand\oprocend{\relax\ifmmode\else\unskip\hfill\fi\oprocendsymbol}

\renewcommand{\theenumi}{(\roman{enumi})}

\renewcommand{\baselinestretch}{0.98}

\begin{document}
\title{Minimum-gain Pole Placement with Sparse Static Feedback}
\author{Vaibhav Katewa, \IEEEmembership{Member, IEEE}, and Fabio Pasqualetti, \IEEEmembership{Member, IEEE}
\thanks{This work was supported in part by awards
    ARO-71603NSYIP and AFOSR-FA9550-19-1-0235.}
\thanks{V. Katewa was with the Department of Mechanical
    Engineering, University of California at Riverside, Riverside 92521 USA. He is now with the Department of Electrical Communication Engineering and the Robert Bosch Center for Cyber-Physical Systems, Indian Institute of Science, Bengaluru 560012, India (e-mail: \href{mailto:vkatewa@iisc.ac.in}{vkatewa@iisc.ac.in}).}
\thanks{F. Pasqualetti is with the Department of Mechanical
    Engineering, University of California at Riverside, Riverside 92521 USA (e-mail: \href{mailto:fabiopas@engr.ucr.edu}{fabiopas@engr.ucr.edu}).}}

%
%
\maketitle

\begin{abstract} The minimum-gain eigenvalue assignment/pole placement
  problem (MGEAP) is a classical problem in LTI systems with static
  state feedback. In this paper, we study the MGEAP when the state
  feedback has arbitrary sparsity constraints. We formulate the sparse
  MGEAP problem as an equality-constrained optimization problem and
  present an analytical characterization of its locally optimal
  solution in terms of eigenvector matrices of the closed loop
  system. This result is used to provide a geometric interpretation of
  the solution of the non-sparse MGEAP, thereby providing additional
  insights for this classical problem. Further, we develop an
  iterative projected gradient descent algorithm to obtain local
  solutions for the sparse MGEAP using a parametrization based on the
  Sylvester equation. We present a heuristic algorithm to compute the
  projections, which also provides a novel method to solve the sparse
  EAP. Also, a relaxed version of the sparse MGEAP is presented and an
  algorithm is developed to obtain approximately sparse local
  solutions to the MGEAP. Finally, numerical studies are presented to
  compare the properties of the algorithms, which suggest that the
  proposed projection algorithm converges in most cases.
\end{abstract}


\section{Introduction}
The Eigenvalue/Pole Assignment Problem (EAP) using static state
feedback is one of the central problems in the design of Linear Time
Invariant (LTI) control systems (e.g., see
\cite{RS-LN-TN-AP:14,YP:16a}). It plays a key role in system
stabilization and shaping its transient behavior. Given the following
LTI system
\begin{subequations}
\begin{align} \label{eq:open_loop1}
\mathcal{D}x(k) &= Ax(k) +Bu(k), \\  \label{eq:open_loop2}
u(k) &= Fx(k),
\end{align}
\end{subequations}
where $x\in\real^{n}$ is the state of the LTI system, $u\in\real^{m}$
is the control input, $A\in\real^{n\times n}$,
$B\in\real^{n\times m}$, and $\mc{D}$ denotes either the continuous
time differential operator or the discrete-time shift operator, the
EAP involves finding a real feedback matrix $F\in\real^{m\times n}$
such that the eigenvalues of the closed loop matrix
$A_c(F)\triangleq A+BF$ coincide with a given set
$\mc{S} = \{\lambda_1,\lambda_2,\cdots,\lambda_n\}$ that is closed
under complex conjugation.

It is well known that the existence of $F$ depends on the
controllability properties of the pair $(A,B)$. Further, for single
input systems ($m=1$), the feedback vector that assigns the
eigenvalues is unique and can be obtained using the Ackermann's
formula \cite{PJA-ANM:05}. On the other hand, for multi-input systems
($m>1$), the feedback matrix is not unique and there exists a
flexibility to choose the eigenvectors of the closed loop system. This
flexibility can be utilized to choose a feedback matrix that satisfies
some auxiliary control criteria in addition to assigning the
eigenvalues. For instance, the feedback matrix can be chosen to
minimize the sensitivity of the closed loop system to perturbations in
the system parameters, thereby making the system robust. This is known
as Robust Eigenvalue Assignment Problem (REAP) \cite{JK-NKN-PVD:85}.
Alternatively, one can choose the feedback matrix with minimum gain,
thereby reducing the overall control effort. This is known as Minimum
Gain Eigenvalue Assignment Problem (MGEAP) \cite{AP-RS-TN:15,YJP:16}.


Recently, considerable attention has been given to the study and
design of sparse feedback control systems, where certain entries of
the matrix $F$ are required to be zero. Feedback sparsity typically
arises in decentralized control problems for large scale and
interconnected systems with multiple controllers \cite{BB-FP-MAD:02},
where each controller has access to only some partial states of the
system. Such constraints in decentralized control problems are
typically specified by information patterns that govern which
controllers have access to which states of the system
\cite{MR-SL:06,AM-NCM-MCR-SY:12}. Sparsity may also be a result of the
special structure of a centralized control system which prohibits
feedback from some states to the controllers.

The feedback design problem with sparsity constraints is considerably more difficult than the unconstrained case. There have been numerous studies to determine the optimal feedback control law for H2/LQR/LQG control problems with sparsity, particularly when the controllers have access to only local information (see \cite{MR-SL:06, BB-FP-MAD:02,FL-MF-MRJ:11al,AM-NCM-MCR-SY:12} and the references therein).
While the optimal H2/LQR/LQG design problems with sparsity have a rich
history, studies on the REAP/MGEAP in the presence of \emph{arbitrary}
sparsity constraints are lacking. Even the problem of finding a
particular (not necessary optimal) sparse feedback matrix that solves
the EAP is not well studied.  In this paper, we study the EAP and
MGEAP with \emph{arbitrary} sparsity constraints on the feedback
matrix $F$. We provide analytical characterization for the solution of
sparse MGEAP and provide iterative algorithms to solve the sparse EAP
and MGEAP. We also briefly discuss the feasibility of the sparse EAP
problem.

\noindent \textbf{Related work} There have been numerous studies on
the optimal pole placement problem \textit{without} sparsity
constraints. For the REAP, authors have considered optimizing
different metrics which capture the sensitivity of the eigenvalues,
such as the condition number of the eigenvector matrix
\cite{JK-NKN-PVD:85, RS-AP-TN:14, MAR-SEF-AB-FT:09, RB-SGN:89,
  AV:00a}, departure from normality \cite{EKC:07} and others
\cite{ALT-YY:96, EKC:01}. Most of these methods use gradient-based
iterative procedures to obtain the solutions. For surveys and
comparisons of these REAP methods, see \cite{RS-AP-TN:14,
  AP-RS-TN-YY-VS-ALT:14, BAW:95} and the references therein.

Early works for MGEAP, including \cite{LFG-DJ:75, BK-RC:80}, presented
approximate solutions using low rank feedback and successive pole
placement techniques. Simultaneous robust and minimum gain pole
placement were studied in \cite{HKT-JL:97,AV:00a, AV:03, AV:00b}. For
a survey and performance comparison of these MGEAP studies, see
\cite{AP-RS-TN:15} and the references therein. The regional pole
placement problem was studied in \cite{SD-BC-DC:10}, \cite{SD-DC:14},
where the eigenvalues were assigned inside a specified region. While
these studies have provided useful insights on REAP/MGEAP, they do not
consider sparsity constraints on the feedback matrix. 
In contrast, we study the sparse EAP/MGEAP by explicitly including the
sparsity constraints in the problem formulation and solutions.

There have also been numerous studies on EAP with sparse
\emph{dynamic} LTI feedback. The concept of decentralized fixed modes
(DFMs) was introduced in \cite{SHW-EJD:73} and later refined in
\cite{JPC-ASM:76}. Decentralized fixed modes are those eigenvalues of
the system which cannot be shifted using a static/dynamic feedback
with \emph{fully decentralized} sparsity pattern (i.e. the case where
controllers have access only to local states). The remaining
eigenvalues of the system can be arbitrarily assigned. However, this
cannot be achieved in general using a static decentralized controller
and requires the use of \emph{dynamic} decentralized controller
\cite{SHW-EJD:73}. Other algebraic characterizations of the DFMs were
presented in \cite{BDOA-DJC:81, EJD-UO:83}. The notion of DFMs was
generalized for an \emph{arbitrary} sparsity pattern and the concept
of structurally fixed modes (SFMs) was introduced in
\cite{MES-DDS:81}. Graph theoretical characterizations of structurally
fixed modes were provided in \cite{VP-MES-DDS:84, VP-MES-DDS:83}. As
in the case of DFMs, assigning the non-SFMs also requires
\emph{dynamic} controllers.  These studies on DFMs and SFMs present
feasibility conditions and analysis methods for the EAP problem with
sparse \emph{dynamic} feedback. In contrast, we study both EAP and
MGEAP with sparse \emph{static} controllers, assuming the sparse EAP
is feasible. We remark that EAP with sparsity and static feedback
controller is in fact important for several network design and control
problems, and easier to implement than its dynamic counterpart.

Recently, there has been a renewed interest in studying linear systems
with sparsity constraints. Using a different approach than
\cite{VP-MES-DDS:84}, the original results regarding DFMs in
\cite{SHW-EJD:73} were generalized for an arbitrary sparsity pattern
by the authors in \cite{AA-MR:14a, AA-MR:14b}, where they also present
a sparse \emph{dynamic} controller synthesis algorithm. Further, there
have been many recent studies on minimum cost input/output and
feedback sparsity pattern selection such that the system has
controllability \cite{SP-GR-SK-APA-JR:17} and no structurally fixed
modes (see \cite{SP-SK-APA:16, SM-PC-MNB:18} and the references
therein). In contrast, we consider the problem of finding a
\emph{static} minimum gain feedback with a \emph{given} sparsity
pattern that solves the EAP.

\noindent\textbf{Contribution} The contribution of this paper is
three-fold. First, we study the MGEAP with static feedback and
arbitrary sparsity constraints (assuming feasibility of sparse
EAP). We formulate the sparse MGEAP as an equality constrained
optimization problem and present an analytical characterization of an
locally optimal sparse solution. As a minor contribution, we use this
result to provide a geometric insight for the non-sparse MGEAP
solutions. Second, we show that determining the feasibility of the
sparse EAP is NP-hard and present necessary and sufficient conditions
for feasibility. We develop two heuristic iterative algorithms to
obtain a local solution of the sparse EAP. The first algorithm is
based on repeated projections on linear subspaces. The second
algorithm is developed using the Sylvester equation based
parametrization and it obtains a solution via projection of a
non-sparse feedback matrix on the space of sparse feedback matrices
that solve the EAP. Third, using the latter EAP projection algorithm,
we develop a projected gradient descent method to obtain a local
solution to the sparse MGEAP. We also formulate a relaxed version of
the sparse MGEAP using penalty based optimization and develop an
algorithm to obtain approximately-sparse local solutions.

\noindent\textbf{Paper organization} The remainder of the
paper is organized as follows. In Section \ref{sec: model} we
formulate the sparse MGEAP optimization problem. In Section
\ref{sec:soln_opt_prob}, we obtain the solution of the optimization
problem using the Lagrangian theory of optimization. We also provide a
geometric interpretation for the optimal solutions of the non-sparse
MGEAP. In Section \ref{sec:algorithms}, we present two heuristic
algorithms for solving the sparse EAP. Further, we present a projected
gradient descent algorithm to solve the sparse MGEAP and also an
approximately-sparse solution algorithm for a relaxed version of the
sparse MGEAP. Section \ref{sec:num_study} contains numerical studies
and comparisons of the proposed algorithms. In Section \ref{sec:
  feasibility}, we discuss the feasibility of the sparse EAP. Finally,
Section~\ref{sec:conclusion} concludes the paper.

\section{Sparse MGEAP formulation}\label{sec: model} 

\subsection{Mathematical notation and preliminary properties}

We use the following properties to derive our results \cite{KBP-MSP:12,JRM-HN:99}:
\begin{enumerate} [label=P.\arabic*, align=left]
\item $\text{tr}(A) = \text{tr}(A^\trans) \: \text{ and } \:\text{tr}(ABC) = \text{tr}(CAB)$, \label{prop:trace}
\item $\Vert A \Vert_{ \text{F}}^{2} = \text{tr}(A^\trans A) = \text{vec}^\trans(A)\text{vec}(A)$, \label{prop:frob}
\item $\text{vec}(AB) = (I\otimes A)\text{vec}(B) = (B^{\trans} \otimes I)\text{vec}(A)$, \label{prop:vec1}
\item $\text{vec}(ABC) = (C^{\trans}\otimes A)\text{vec}(B)$, \label{prop:vec2}
\item $(A\otimes B)^{\trans} = A^{\trans} \otimes B^{\trans}$ and $(A\otimes B)^{\conjtrans} = A^{\conjtrans} \otimes B^{\conjtrans}$, \label{prop:kron}
\item $1_n^\trans(A\circ B)1_n = \text{tr}(A^\trans B)$, \label{prop:lag_mult}
\item $A\circ B = B\circ A \: \text{ and } \: A\circ (B\circ C) = (A\circ B)\circ C$, \label{prop:had1}
\item $\text{vec}(A\circ B) = \text{vec}(A)\circ \text{vec}(B),  (A\circ B)^\trans = A^\trans\circ B^\trans$, \label{prop:had2} 
\item $\frac{d}{dX}\text{tr}(AX) \!= \!A^\trans$, $\frac{d}{dX}\text{tr}(X^\trans X) \!=\! 2X$, $\frac{d}{dx}(Ax)\! =\! A$, \label{prop:der} \vspace{3pt}
\item $d(X^{-1}) = -X^{-1} dX X^{-1}$, \label{prop:diff_inv}
\item Let $D_xf$ and $D^2_xf$ be the gradient and Hessian of\\
 \hspace*{13pt}$f(x): \mathbb{R}^{n}\rightarrow \mathbb{R}$. Then, $df = (D_xf) ^{\trans} dx$ and \\
 \hspace*{10pt} $d^2f = (dx)^{\trans}(D^{2}_xf) dx$, \label{prop:grad_Hess}
 \item Projection of a vector $y\in\real^{n}$ on the null space of \hspace*{12pt} $A\in\real^{m\times n}$ is given by $y_p = [I_n-A^{+}A]y$. \label{prop:projection}
\end{enumerate}

The Kronecker sum of two square matrices $A$ and $B$ with dimensions $n$ amd $m$, respectively, is denoted by
\begin{align*}
A\oplus B = (I_m \otimes A) + B\otimes I_n.
\end{align*}

Further, we use the following notation throughout the paper:

\begin{tabular} {|c|l|}
\hline
$\Vert\cdot\Vert_2$ & Spectral norm \\
$\Vert\cdot\Vert_\text{F}$ & Frobenius norm\\
$<\cdot,\cdot>_F$ & Inner (Frobenius) product\\ 
$\lvert \cdot \rvert$ & Cardinality of a set \\

$\Gamma(\cdot)$ & Spectrum of a matrix\\
$\sigma_{min}(\cdot)$ & Minimum singular value of a matrix\\
$\text{tr}(\cdot )$  & Trace of a matrix \\
$(\cdot)^{+}$  & Moore-Penrose pseudo inverse \\
$(\cdot)^{\trans}$ & Transpose of a matrix\\
$\mathcal{R}(\cdot)$ & Range of a matrix\\ 
$A>0$ & Positive definite matrix $A$\\

$\circ$ & Hadamard (element-wise) product \\
$\otimes $ & Kronecker product \\
$(\cdot)^{*}$ & Complex conjugate\\
$(\cdot)^{\conjtrans}$ &  Conjugate transpose \\

$\text{supp}(\cdot)$ & Support of a vector \\
$\text{vec}(\cdot)$ & Vectorization of a matrix\\
$\text{diag}(a)$ & $n\times n$ Diagonal matrix with diagonal\\ 
 & elements given by $n$-dim vector $a$\\
$\text{Re}(\cdot)$ & Real part of a complex variable \\
$\text{Im}(\cdot)$ & Imaginary part of a complex variable\\

$1_n (0_n)$ & $n$-dim vector of ones (zeros)\\
$1_{n\times m}(0_{n\times m})$ & $n\times m$-dim matrix of ones (zeros)\\
$I_n$ & $n$-dim identity matrix\\
$e_i$ & $i$-th canonical vector\\
$T_{m,n}$ & Permutation matrix that satisfies\\
&  $\text{vec}(A^{\trans}) = T_{m,n}\text{vec}(A)$, $A\in\real^{m\times n}$\\
\hline
\end{tabular}

\vspace{4pt}

\subsection{Sparse MGEAP}
The sparse MGEAP involves finding a real feedback matrix
$F\in\real^{m\times n}$ with minimum norm that assigns the closed loop
eigenvalues of \eqref{eq:open_loop1}-\eqref{eq:open_loop2} at some
desired locations given by set
$\mc{S} = \{\lambda_1,\lambda_2,\cdots,\lambda_n\}$, and satisfies a
given sparsity constraints. Let
$\F\in\{0,1\}^{m\times n}$ denote a binary matrix that specifies the
sparsity structure of the feedback matrix $F$. If $\F_{ij} = 0$
(respectively $\F_{ij}=1$), then the $j^{\text{th}}$ state is
unavailable (respectively available) for calculating the
$i^{\text{th}}$ input. Thus,
\begin{align*}
F_{ij} = 
\begin{cases}
0  \quad \text{if} \:\:\F_{ij} = 0, \:\:\text{and} \\
\star \quad \text{if} \:\:\F_{ij} = 1,
\end{cases}
\end{align*}
where $\star$ denotes a real number. Let $\F^c\triangleq 1_{m\times n}-\F$ denote the complementary sparsity structure matrix. Further, (with a slight abuse of notation, c.f. \eqref{eq:open_loop1}) let $X\triangleq[x_1,x_2,\cdots,x_n]\in\mathbb{C}^{n\times n}, x_i\neq 0_n$ denote the non-singular eigenvector matrix of the closed loop matrix $A_c(F)=A+BF$.

The MGEAP can be mathematically stated as follows:
\begin{align} \label{eq:opt_cost}
\underset{F,X}{\min} &\quad \frac{1}{2}\:||F||_F^2  \hspace{110pt}
\end{align}
\ajustspaceandequationnumber
\begin{subequations}
\begin{align}  \label{eq:eigv_assgn}
\text{s.t.}& \quad (A+BF)X=X\Lambda, \\  \label{eq:spar_const}
 &\quad \F^c \circ F = 0_{m\times n},
\end{align}
\end{subequations}
where
$\Lambda = \text{diag}([\lambda_1,\lambda_2, \cdots,
\lambda_n]^{\trans})$ is the diagonal matrix of the desired
eigenvalues. Equations \eqref{eq:eigv_assgn} and \eqref{eq:spar_const}
represent the eigenvalue assignment and sparsity constraints,
respectively.

The constraint
  \eqref{eq:eigv_assgn} is not convex in $(F,X)$ and, therefore, the
  optimization problem \eqref{eq:opt_cost} is
  non-convex. Consequently, multiple local minima may exist. This is a
  common feature in various minimum distance and eigenvalue assignment
  problems \cite{DK-MV:15}, including the non-sparse MGEAP. 

  \begin{remark} (\textbf{Choice of norm}) The Frobenius norm measures
    the element-wise gains of a matrix, which is informative in
    sparsity constrained problems arising, for instance, in network
    control problems. It is also convenient for the analysis,
    particularly to compute the derivatives of the cost
    function. \oprocend
\end{remark}



\begin{definition}\textbf{(Fixed modes \cite{AA-MR:14a,SHW-EJD:73})} 
  The fixed modes of $(A,B)$ with respect to the sparsity constraints
  $\F$ are those eigenvalues of $A$ which cannot be changed using LTI
  static (and also dynamic) state feedback, and are denoted by
  \begin{align*}
    \Gamma_{f}(A,B,\F) \triangleq \underset{\textstyle F : F \circ
    \F^{c}=0}{\bigcap} \Gamma(A+BF). 
  \end{align*}  
\end{definition}

We make the following assumptions regarding the fixed modes and
feasibility of the optimization problem \eqref{eq:opt_cost}.

\begin{assumption} \label{assump:fix_mode_in_eigv_set}
The fixed modes of the triplet $(A,B,\F)$ are
included in the desired eigenvalue set $\mc{S}$, i.e.,
$\Gamma_{f}(A,B,\F) \subseteq \mc{S}$.
\end{assumption}

\begin{assumption} \label{assump:feasibility}
There exists at least one feedback matrix $F$
that satisfies constraints \eqref{eq:eigv_assgn}-\eqref{eq:spar_const}
for the given $\mc{S}$.
\end{assumption}

Assumption \ref{assump:fix_mode_in_eigv_set} is clearly necessary for the feasibility of the
optimization problem \eqref{eq:opt_cost}. Assumption \ref{assump:feasibility} is
restrictive because, in general, it is possible that a static feedback
matrix with a given sparsity pattern cannot assign the closed loop
eigenvalues to arbitrary locations (i.e. for an arbitrary set $\mc{S}$
satisfying Assumption \ref{assump:fix_mode_in_eigv_set})\footnote{Note that a sparse
  \emph{dynamic} feedback law can assign the eigenvalues to arbitrary
  locations under Assumption \ref{assump:fix_mode_in_eigv_set} \cite{AA-MR:14b}.}. In such
cases, only a few ($<n$) eigenvalues can be assigned independently and
other remaining eigenvalues are a function of them. To the best of our
knowledge, there are no studies on characterizing conditions for the
existence of a static feedback matrix for an arbitrary sparsity
pattern $\F$ and eigenvalue set $\mc{S}$ \cite{JR-JCW:99} (although
such characterization is available for \textit{dynamic} feedback laws
with arbitrary sparsity pattern \cite{MES-DDS:81, AA-MR:14a,
  AA-MR:14b}, and static output feedback for \textit{decentralized}
sparsity pattern \cite{JL-NK:95}). Thus, for the purpose of this
paper, we focus on finding the optimal feedback matrix
\textit{assuming} that at least one such feedback matrix exists. We
provide some preliminary results on the feasibility of the
optimization problem \eqref{eq:opt_cost} in Section~\ref{sec:
  feasibility}.

\section{Solution to the sparse MGEAP}\label{sec:soln_opt_prob} In
this section we present the solution to the optimization problem
\eqref{eq:opt_cost}. To this aim, we use the theory of Lagrangian
multipliers for equality constrained minimization problems.


\begin{remark} (\textbf{Conjugate
    eigenvectors}) \label{rem:conj_eigvec} We use the convention that
  the right (respectively, left) eigenvectors $(x_i,x_j)$
  corresponding to two conjugate eigenvalues $(\lambda_i,\lambda_j)$
  are also conjugate. Thus, if $\lambda_i = \lambda_j^{*}$, then
  $x_i=x_j^{*}$. \oprocend
\end{remark}

We use the real counterpart of \eqref{eq:eigv_assgn} for the analysis. For two complex conjugate eigenvalues $(\lambda_i,\lambda_i^{*})$ and corresponding eigenvectors $(x_i,x_i^{*})$, the following complex equation
\begin{align*}
(A+BF)\begin{bmatrix}x_i &x_i^{*}\end{bmatrix} = \begin{bmatrix}x_i & x_i^{*}\end{bmatrix} \left[\begin{smallmatrix} \lambda_i & 0\\ 0 & \lambda_i^{*} \end{smallmatrix}\right]
\end{align*}
 is equivalent to the following real equation
\begin{align*}
(A\!+\!BF)\!
\begingroup 
\setlength\arraycolsep{2pt}
\begin{bmatrix}\text{Re}(x_i) & \text{Im}(x_i)\end{bmatrix}\! \!=
                                \!\!\begin{bmatrix}\text{Re}(x_i)
                                  &\text{Im}(x_i)\end{bmatrix}\!\!
                                    \left[\begin{smallmatrix}
                                        \phantom{-}\text{Re}(\lambda_i)
                                        & \text{Im}(\lambda_i)\\
                                        -\text{Im}(\lambda_i) &
                                        \text{Re}(\lambda_i) \end{smallmatrix}\right] \endgroup 
                                                                .
\end{align*}
For each complex eigenvalue, the columns
$\begin{bmatrix}x_i & x_i^{*}\end{bmatrix}$ of $X$ are replaced by
$\begin{bmatrix}\text{Re}(x_i) & \text{Im}(x_i)\end{bmatrix}$ to
obtain a real $X_R$, and the sub-matrix
$\left[\begin{smallmatrix} \lambda_i & 0\\ 0 &
    \lambda_i^{*} \end{smallmatrix}\right]$ of $\Lambda$ is replaced
by
$\left[\begin{smallmatrix} \phantom{-}\text{Re}(\lambda_i) &
    \text{Im}(\lambda_i)\\ -\text{Im}(\lambda_i) &
    \text{Re}(\lambda_i) \end{smallmatrix}\right]$ to obtain a real
$\Lambda_R$. The real eigenvectors in $X$ and $X_R$, and real
eigenvalues in $\Lambda$ and $\Lambda_R$ coincide. Clearly, $X_R$ is
not the eigenvector matrix of $A+BF$ (c.f. Remark
\ref{rem:eigstr_assgn}), and $X$ can be obtained through the columns
of $X_R$. Thus, \eqref{eq:eigv_assgn} becomes
\begin{align} \label{eq:eigv_assgn_real}
(A+BF)X_R=X_R\Lambda_R,
\end{align}
and $X_R$ replaces the optimization variable $X$ in
\eqref{eq:opt_cost}. In the theory of equality constrained
optimization, the first-order optimality conditions are meaningful
only when the optimal point satisfies the following regularity
condition: the Jacobian of the constraints, defined by $J_b$, is full
rank. This regularity condition is mild and usually satisfied for most
classes of problems \cite{DGL-YY:08}.  Before presenting the main
result, we derive the Jacobian and state the regularity condition for
the problem~\eqref{eq:opt_cost}.

Computation of $J_b$ requires vectorization of the matrix constraints
\eqref{eq:eigv_assgn_real} and \eqref{eq:spar_const}. For this
purpose, let $x_R\triangleq \text{vec}(X_R)\in\mathbb{C}^{n^2}$,
$f\triangleq \text{vec}(F)\in\mathbb{R}^{mn}$, and let
$z\triangleq [x_R^{\trans},f^{\trans}]^{\trans}$ be the vector
containing all the independent variables of the optimization problem.
Further, let $n_s$ denote the total number of feedback sparsity
constraints (i.e. number of $1$'s in $\F^c$):
\begin{align*}
  \subscr{n}{s} &= | \setdef{ (i,j) }{ \F^c  = [\bar{\mathsf{f}}^c_{ij}],\, \bar{\mathsf{f}}^c_{ij} = 1 } |.
\end{align*}
Note that the constraint \eqref{eq:spar_const}  consists of $n_s$ non-trivial sparsity constraints, and can be
equivalently written as 
\begin{align}\label{eq:spar_const_vec}
  Qf = 0_{n_s},
\end{align}
where $Q=\begin{bmatrix} e_{q_1} & e_{q_2} & \dots & e_{q_{n_{s}}} \end{bmatrix}^\trans \in\{0,1\}^{n_s\times mn}$ with $\{q_1, \dots, q_{n_{s}}\} = \text{supp}(\text{vec}( \F^c))$ being the set of indices indicating the ones in $\text{vec}(\F^c)$.

\begin{lemma}{\bf \emph{(Jacobian of the constraints)}}\label{lem:jacobian} The Jacobian of the equality constraints \eqref{eq:eigv_assgn}-\eqref{eq:spar_const} is given by
\begin{align}\label{eq:jacobian}
 J_b(z)= \begin{bmatrix}
{A}_c(F)\!\oplus\! (\!-\Lambda_R^{\trans}) &  X_R^{\trans}\!\otimes\! B \\
0_{n_s\times n^2 } & Q
\end{bmatrix}.
\end{align}
\end{lemma}
\begin{proof}
  We construct the Jacobian $J_b$ by rewriting the constraints
  \eqref{eq:eigv_assgn_real} and \eqref{eq:spar_const},
   in vectorized form and taking their
  derivatives with respect to $z$. Constraint \eqref{eq:eigv_assgn_real}
  can be vectorized in the following two different ways (using
  \ref{prop:vec1} and \ref{prop:vec2}):
  \begin{subequations}
    \begin{align} \label{eq:eigv_assgn_vec1}
[(A+BF) \oplus (-\Lambda_R^{\trans})]x_R = 0_{n^2}, \\ \label{eq:eigv_assgn_vec2}
[A \oplus -(\Lambda_R^{\trans})]x_R + (X_R^{\trans}\otimes B)f = 0_{n^2}.
\end{align}
\end{subequations}
Differentiating \eqref{eq:eigv_assgn_vec1} w.r.t. $x_R$ and \eqref{eq:eigv_assgn_vec2} w.r.t $f$ yields the first (block) row of $J_b$. 
Differentiating \eqref{eq:spar_const_vec} w.r.t. $z$ yields the second
(block) row of $J_b$, thus completing the proof.
\end{proof}

We now state the optimality conditions for the
problem~\eqref{eq:opt_cost}.

\begin{theorem}{\bf \emph{(Optimality conditions)}}\label{thm:opt_feedback}
  Let $(\hat{X}, \hat{F})$ (equivalently
  $\hat{z} = [\hat{x}_R^{\trans},\hat{f}^{\trans}]^{\trans}$)
  satisfy the constraints
  \eqref{eq:eigv_assgn}-\eqref{eq:spar_const}. Let
  $\hat{L} = [\hat{l}_i]$, $i=1, \cdots,n$ be the left eigenvector
  matrix of $A_c(\hat{F})$, and let $\hat{L}_R$ be its real
  counterpart constructed by replacing $[\hat{l}_i,\hat{l}_i^{*}]$
  with $[\textup{Re}(\hat{l}_i) , -\textup{Im}(\hat{l}_i)]$. Let
  $J_b(z)$ be defined in Lemma \ref{lem:jacobian} and
  $P(z) = I_{n^2+mn}-J_b^{+}(z)J_b(z)$. Further, define
  $\bar{L} \triangleq 4 T_{n,m}(B^{\trans}\hat{L}_R\otimes I_n)$ and
  let
\begin{align} \label{eq:hessian}
\hat{D} \triangleq \begin{bmatrix} 
0_{n^2\times n^2} & \bar{L}^{\trans} \\
 \bar{L} &  2I_{mn} 
\end{bmatrix}.
\end{align}
Then, $(\hat{X},\hat{F})$ is a local minimum of the optimization problem \eqref{eq:opt_cost}
if and only if
\begin{subequations}
\begin{align} \label{eq:opt_cond1}
\hat{F} = -\F \circ (B^{\trans} \hat{L} \hat{X}^{\trans}), \\
  (A+B\hat{F})\hat{X}=\hat{X}\Lambda   \label{eq:right_eigvec} \\ \label{eq:left_eigvec}
  (A+B\hat{F})^{\trans} \hat{L} = \hat{L} \Lambda \\  \label{eq:regularity}
  J_b(\hat{z})\:\: \text{is full rank,}   \\ \label{eq:pos_definite}
  P(\hat{z}) \hat{D} P(\hat{z}) > 0.
\end{align}
\end{subequations}
\end{theorem}

\begin{proof}
  We prove the result using the Lagrange theorem for equality
  constrained minimization. Let $L_R\in\mathbb{R}^{n\times n}$ and
  $M\in\mathbb{R}^{m\times n}$ be the Lagrange multipliers associated
  with constraints \eqref{eq:eigv_assgn_real} and
  \eqref{eq:spar_const}, respectively. The Lagrange function for the
  optimization problem \eqref{eq:opt_cost} is given by
  \begin{align*}
    \mc{L} \overset{\ref{prop:frob}}{=}  & \frac{1}{2}\:\text{tr}(F^{\trans}F) + 2 \: 1_n^{\trans}[L_R\circ(A_c(F)X_R-X_R\Lambda_R)]1_n  \\
    + 1_m^{\trans}&[M\circ(\F^c \circ F) ]1_n \\
    \overset{\ref{prop:lag_mult},\ref{prop:had1}}{=} & \frac{1}{2}\:\text{tr}(F^{\trans}F) + 2\:  \text{tr}[L_R^{\trans}(A_c(F)X_R-X_R\Lambda_R)]  \\
+ & \text{tr}[(M\circ\F^c)^{\trans} F].
\end{align*}
\emph{Necessity:} We next derive the first-order necessary condition
for a stationary point. Differentiating $\mc{L}$ w.r.t. $X_R$ and
setting to $0$, we get \begin{align} \label{eq:l_der_X}
    \frac{d}{dX_R}\mc{L} \overset{\ref{prop:der}}{=}
    2[A_c^{\trans}(F)L_R-L_R\Lambda_R^{\trans} ] = 0_{n\times n}.
\end{align}
The real equation \eqref{eq:l_der_X} is equivalent to the complex equation \eqref{eq:left_eigvec}.
Equation \eqref{eq:right_eigvec} is a restatement of
\eqref{eq:eigv_assgn} for the optimal
$(\hat{F},\hat{X})$. Differentiating $\mc{L}$ w.r.t. $F$, we get
\begin{align} \label{eq:l_der_K} \frac{d}{dF}\mc{L}
    \overset{\ref{prop:der}}{=} F +2 B^{\trans}L_RX_R^{\trans}+ M\circ
    \F^c = 0_{m\times n}.
\end{align}
Taking the Hadamard product of \eqref{eq:l_der_K} with $\F^c$ and using \eqref{eq:spar_const}, we get (since $\F^{c}\circ \F^{c} = \F^{c}$)
\begin{align} \label{eq:l_der_K_1}
 \F^c \circ (2B^{\trans}L_RX_R^{\trans}) + M\circ \F^c = 0_{m\times n}
\end{align}
Replacing $M\circ \F^c$ from \eqref{eq:l_der_K_1} into \eqref{eq:l_der_K}, we get
\begin{align*}
  F = -\F \circ (2B^{\trans} L_R X_R^{\trans})  \overset{(a)} {=} -\F \circ (B^{\trans} L X^{\trans}),
\end{align*}
where $(a)$ follows from the definition of $L_R$ and $X_R$ and Remark
\ref{rem:conj_eigvec}. Equation \eqref{eq:regularity} is the
necessary regularity condition and follows from Lemma
\ref{lem:jacobian}

\noindent \emph{Sufficiency:} Next, we derive the second-order
sufficient condition for a local minimum by calculating the Hessian of
$\mc{L}$. Taking the differential of $\mc{L}$ twice, we get
\begin{align*} d^2 \mc{L} &= \text{tr}((dF)^{\trans}dF) +4
    \text{tr}(L_R^{\trans}BdFdX_R)
    \\ 
 &\overset{(\ref{prop:frob})}{=}   df^{\trans}df +4 \text{vec}^{\trans}(dF^{\trans}B^{\trans}L_R)dx_R \\
 &\overset{(\ref{prop:vec1},\ref{prop:kron})}{=}  df^{\trans}df+ df^{\trans}\bar{L}dx \\
 &= \frac{1}{2} \begin{bmatrix} dx^{\trans} &df^{\trans} \end{bmatrix} D \begin{bmatrix} dx\\df \end{bmatrix},
\end{align*}
where $D$ is the Hessian (c.f. \ref{prop:grad_Hess}) defined in
\eqref{eq:hessian}.  The sufficient second-order optimality condition
for the optimization problem requires the Hessian to be positive
definite in the kernel of the Jacobian at the optimal point
\cite[Chapter 11]{DGL-YY:08}. That is,
$y^\trans D y > 0, \; \forall y: J_b(z)y = 0$. This condition is equivalent to $P(z) D P(z) > 0$, since $J_b(z)y = 0$ if and only if $y = P(z)s$ for a $s\in \real^{n^2+mn}$ \cite{DGL-YY:08}. Since the projection matrix $P(z)$ is symmetric, \eqref{eq:pos_definite} follows, and this concludes the proof.
\end{proof}


Observe that the Hadamard product in \eqref{eq:opt_cond1} guarantees
that the feedback matrix satisfies the sparsity constraints given in
\eqref{eq:spar_const}. However, the optimal sparse feedback $\hat{F}$
cannot be obtained by sparsification of the optimal non-sparse
feedback. 
The optimality condition \eqref{eq:opt_cond1} is an implicit condition
in terms of the closed loop right and left eigenvector matrices. Next,
we provide an explicit optimality condition in terms of
$\{\hat{L},\hat{X}\}$.

\begin{corollary} (\textbf{Stationary point of
    \eqref{eq:opt_cost}}) \label{cor:stat_pt}
  $\hat{Z} \triangleq [\hat{X}^{\trans},\hat{L}^{\trans}]^{\trans}$ is
  a stationary point of the optimization problem \eqref{eq:opt_cost}
  if and only if

\begin{align}    \label{eq:stationary_soln}
\bar{A}\hat{Z} - \hat{Z} \Lambda = \bar{B}_1[\mathsf{F}\circ(\bar{B}_1^{\trans} \bar{I} \hat{Z} \hat{Z}^{\trans} \bar{B}_2)]\bar{B}_2^{\trans} \hat{Z},
\end{align}
where,
\begin{align*}
\bar{A} &\triangleq \begin{bmatrix} A & 0_{n\times n} \\ 0_{n\times n} & A^{\trans}\end{bmatrix}, \:\bar{B}_1 \triangleq \begin{bmatrix} B & 0_{n\times n} \\ 0_{n\times m} & I_n\end{bmatrix}, \\
\bar{B}_2 &\triangleq \begin{bmatrix} I_n & 0_{n\times m} \\ 0_{n\times n} & B\end{bmatrix}, \:\:  \mathsf{F} \triangleq \begin{bmatrix} \F & 0_{m\times m} \\ 0_{n\times n} & \F^{\trans}\end{bmatrix}, \text{and}\\
\bar{I} & \triangleq \begin{bmatrix} 0_{n\times n} & I_n \\ I_n & 0_{n\times n}\end{bmatrix}.
\end{align*}
\end{corollary}
\begin{proof}
Combining \eqref{eq:right_eigvec} and \eqref{eq:left_eigvec} and using $\Lambda^{\trans} = \Lambda$, we get
\begin{align*}
&\begin{bmatrix} A\hat{X} -\hat{X}\Lambda \\  A^{\trans}\hat{L} -\hat{L}\Lambda \end{bmatrix} =-\begin{bmatrix} B\hat{F}\hat{X}\\ \hat{F}^{\trans}B^{\trans}\hat{L}    \end{bmatrix} \\
&\Rightarrow \bar{A}\hat{Z} - \hat{Z}\Lambda = -\bar{B}_1 \begin{bmatrix} \hat{F} & 0 \\ 0 & \hat{F}^{\trans} \end{bmatrix} \bar{B}_2^{\trans} \hat{Z} \\
 &= \bar{B}_1 \begin{bmatrix} \F \circ (B^{\trans} \hat{L} \hat{X}^{\trans}) & 0 \\ 0 & \F^{\trans} \circ (\hat{X}\hat{L} ^{\trans}B) \end{bmatrix} \bar{B}_2^{\trans} \hat{Z} \\
& = \bar{B}_1 \left(\mathsf{F} \circ \left\{\bar{B}_1^{\trans}  \begin{bmatrix} \hat{L}\hat{X}^{\trans} & 0 \\ 0 & \hat{X}\hat{L}^{\trans}\end{bmatrix}\bar{B}_2 \right\}    \right)\bar{B}_2^{\trans} \hat{Z} \\
& = \bar{B}_1 \left(\mathsf{F} \circ \left\{\bar{B}_1^{\trans}  \bar{I}(\hat{Z} \hat{Z}^{\trans} \circ \bar{I})\bar{B}_2 \right\}    \right)\bar{B}_2^{\trans} \hat{Z} \\
& =  \bar{B}_1\{\mathsf{F}\circ(\bar{B}_1^{\trans} \bar{I} \hat{Z} \hat{Z}^{\trans} \bar{B}_2)\}\bar{B}_2^{\trans} \hat{Z},
\end{align*}
where the equalities follow from the Hadamard product.
\end{proof}

\begin{remark} (\textbf{Partial spectrum
    assignment}) \label{rem:partial_assgn} The results of Theorem
  \ref{thm:opt_feedback} and Corollary \ref{cor:stat_pt} are also
  valid when specifying only $p<n$ eigenvalues (the remaining
  eigenvalues are functionally related to them; see also the
  discussion below Assumption \ref{assump:feasibility}). In this case,
  $\Lambda\in\mathbb{C}^{p\times p}$,
  $\hat{X}\in\mathbb{C}^{n\times p}$ and
  $\hat{L}\in\mathbb{C}^{n\times p}$. While partial assignment may be
  useful in some applications, in this paper we focus on assigning all
  the eigenvalues. \oprocend
\end{remark}

\begin{remark} \textbf{(General eigenstructure assignment)} \label{rem:eigstr_assgn}
Although the optimization problem \eqref{eq:opt_cost} is formulated by considering $\Lambda$ to be diagonal, the result in Theorem \ref{thm:opt_feedback} is valid for any general $\Lambda$ satisfying $\Gamma(\Lambda) = \mc{S}$. For instance, we can choose $\Lambda$ in a Jordan canonical form. However, note that for a general $\Lambda$, $X$ will cease to be an eigenvector matrix.  \oprocend 
\end{remark}

A solution of the optimization problem \eqref{eq:opt_cost} can be
obtained by numerically/iteratively solving the matrix equation
\eqref{eq:stationary_soln}, which resembles a Sylvester type equation
with a non-linear right side, and using \eqref{eq:opt_cond1} to
compute the feedback matrix. The regularity and local minimum of the
solution can be verified using \eqref{eq:regularity} and
\eqref{eq:pos_definite}, respectively. Since the optimization problem
is not convex, only local minima can be obtained via this
procedure. To improve upon the local solutions, the procedure can be
repeated for different initial conditions to solve
\eqref{eq:stationary_soln}. However, convergence to a global minimum
is not guaranteed.

The convergence of the iterative techniques to solve
\eqref{eq:stationary_soln} depends substantially on the initial
conditions. If they are not chosen properly, convergence may not be
guaranteed. Further, the solution of \eqref{eq:stationary_soln} can
also represent a local maxima. Therefore, instead of solving
\eqref{eq:stationary_soln} directly, we use a different approach based
on the gradient descent procedure to obtain a locally minimum
solution. Details of this approach and corresponding algorithms are
presented in Section \ref{sec:algorithms}.


\subsection{Results for the non-sparse MGEAP}
In this subsection, we present some results specific to the case when
the optimization problem \eqref{eq:opt_cost} does not have any
sparsity constraints (i.e. $\F=1_{m\times n}$). Although the
non-sparse MGEAP has been studied previously, these results are novel
and further illustrate the properties of an optimal
solution.

We begin by presenting a geometric interpretation of the optimality
conditions in Theorem \ref{thm:opt_feedback} with $B=I_n$, i.e. all
the entries of $A$ can be perturbed independently. In this case, the
optimization problem \eqref{eq:opt_cost} can be written
as: 
\begin{align} \label{eq:opt_prob_pert}
 \underset{X}{\min} &\quad \frac{1}{2}\:||A-X\Lambda X^{-1}||_F^2.
\end{align}

Since $A$ and $R(X)\triangleq X\Lambda X^{-1}$ are elements (or
vectors) of the matrix inner product space with Frobenius norm, a
solution of the optimization problem \eqref{eq:opt_prob_pert} is given
by the projection of $A$ on the manifold
$\mathcal{M}\triangleq \{R(X): X \text{ is non-singular}\}$. This
projection can be obtained by solving the normal equation, which
states that the optimal error vector
$\hat{F} = A-\hat{X}\Lambda \hat{X}^{-1}$ should be orthogonal to the
tangent plane of the manifold $\mathcal{M}$ at the optimal point
$\hat{X}$ \cite{PAA-RM-RS:08}. The next result shows that the
optimality conditions derived in Theorem \ref{thm:opt_feedback} are in
fact the normal equations for the optimization problem
\eqref{eq:opt_prob_pert}.

\begin{lemma} (\textbf{Geometric interpretation})
Let $\F = 1_{m\times n}$ and $B = I_n$. Then, Equations \eqref{eq:opt_cond1}-\eqref{eq:left_eigvec} are equivalent to the following normal equation:
\begin{align} \label{eq:normal_eq}
<\hat{F},\mathcal{T}_{\mathcal{M}}(\hat{X})>_F = 0,
\end{align}
where $\mathcal{T}_{\mathcal{M}}(X)$ denotes the tangent space of $\mathcal{M}$ at $X$.
\end{lemma}

\begin{proof}
  We begin by characterizing the tangent space
  $\mathcal{T}_{\mathcal{M}}(X)$, which is given by the first order
  approximation of $R(X)$:
\begin{align*}
R(X+dX)& = (X+dX)\Lambda (X+dX)^{-1} \\
& \overset{(\ref{prop:diff_inv})}{=} R(X) + dX\Lambda X^{-1} - X\Lambda X^{-1} dX X^{-1} \\
& \hspace{42pt} + \text{higher order terms}.
\end{align*}
Thus, the tangent space is given by
\begin{align*}
\mathcal{T}_{\mathcal{M}}(X) = \{ Y \Lambda X^{-1} - X\Lambda X^{-1} Y X^{-1} : Y \in \mathbb{C}^{n\times n}\}
\end{align*}
\textit{Necessity:} Using $\hat{F}$ given by \eqref{eq:opt_cond1}, we get
\begin{align*}
<\hat{F},&\mathcal{T}_{\mathcal{M}}(\hat{X})> = \text{tr}(\hat{F}^{\trans} (Y \Lambda \hat{X}^{-1} - \hat{X}\Lambda\hat{X}^{-1} Y \hat{X}^{-1}))\\
& = -\text{tr}(\hat{X}\hat{L}^{\trans} Y \Lambda \hat{X}^{-1}) + \text{tr}(\hat{X}\hat{L}^{\trans} \hat{X}\Lambda \hat{X}^{-1} Y \hat{X}^{-1}) \\
\overset{(\ref{prop:trace})}{=}& -\text{tr}(\hat{L}^{\trans} Y \Lambda) + \text{tr}(\hat{L}^{\trans} \hat{X}\Lambda \hat{X}^{-1} Y) \\
\overset{(a)}{=} &-\text{tr}(\hat{L}^{\trans} Y \Lambda) + \text{tr}(\Lambda \hat{L}^{\trans} \hat{X} \hat{X}^{-1} Y)  \overset{(\ref{prop:trace})}{=} 0,
\end{align*}
where $(a)$ follows from the fact that $\Lambda$ and $\hat{L}^{\trans}\hat{X}$ commute.

\noindent\textit{Sufficiency:} From \eqref{eq:normal_eq}, we get
\begin{align*}
\text{tr}(\hat{F}^{\trans} (Y \Lambda \hat{X}^{-1} - \hat{X}\Lambda \hat{X}^{-1} Y \hat{X}^{-1})) = 0 \\
\overset{(\ref{prop:trace})}{\Rightarrow} \text{tr}[ (\Lambda \hat{X}^{-1} \hat{F}^{\trans} - \hat{X}^{-1} \hat{F}^{\trans} \hat{X}\Lambda \hat{X}^{-1})Y] = 0.
\end{align*}
Since the above equation is true for all $Y \in \mathbb{C}^{n\times n}$, we get
\begin{align*}
\Lambda \hat{X}^{-1} &\hat{F}^{\trans} - \hat{X}^{-1} \hat{F}^{\trans} \hat{X}\Lambda \hat{X}^{-1} = 0_{n\times n} \\
&\Rightarrow \hat{X} \Lambda \hat{X}^{-1} \hat{F}^{\trans} =  \hat{F}^{\trans} \hat{X}\Lambda \hat{X}^{-1}\\
&\Rightarrow A_c(\hat{F}) \hat{F}^{\trans} =  \hat{F}^{\trans} A_c(\hat{F}).
\end{align*}
Thus, $A_c(\hat{F})$ and $\hat{F}^{\trans}$ commute and have common
left and right eigenspaces \cite{PL:85}, i.e.,
$\hat{F}^{\trans} = -\hat{X} G \hat{X}^{-1} =
-\hat{X}\hat{L}^{\trans}$, where $G$ is a diagonal matrix. This
completes the proof.
\end{proof}

Next, we show the equivalence of the non-sparse MGEAP for two orthogonally similar systems.

\begin{lemma} (\textbf{Invariance under orthogonal transformation}) Let
  $\F = 1_{m\times n}$ and $(A_1,B_1)$, $(A_2,B_2)$ be two
  orthogonally similar systems such that $A_2 = PA_1P^{-1}$ and
  $B_2 = PB_1$, with $P$ being an orthogonal matrix. Let optimal
  solutions of \eqref{eq:opt_cost} for the two systems be denoted by
  $(\hat{X}_1,\hat{L}_1,\hat{F}_1)$ and
  $(\hat{X}_2,\hat{L}_2,\hat{F}_2)$, respectively. Then
  \begin{align}
    \begin{split}
      &\hat{X}_2 = P\hat{X}_1, \; \hat{L}_2 = P\hat{L}_1,  \; \hat{F}_2 =
      \hat{F}_1P^{\trans}, \text{ and }\\
      &||\hat{F}_1||_F = ||\hat{F}_2||_F.
    \end{split}
  \end{align}
\end{lemma}
\begin{proof}
From \eqref{eq:right_eigvec}, we have
\begin{align*}
(A_2+B_2\hat{F_2})\hat{X_2} = \hat{X_2}\Lambda \\
\Rightarrow (PA_1P^{-1}+PB_1\hat{F_1}P^{\trans})P\hat{X_1} = P\hat{X_1}\Lambda\\
\Rightarrow (A_1+B_1\hat{F_1})\hat{X_1} = \hat{X_1}\Lambda.
\end{align*}
Similar relation can be shown between $\hat{L}_1$ and $\hat{L}_2$ using \eqref{eq:left_eigvec}. Next, from \eqref{eq:opt_cond1}, we have
\begin{align*}
\hat{F}_2 &= -B_2^{\trans}\hat{L}_2\hat{X}_2^{\trans}    =       -B_1^{\trans}\hat{L}_1\hat{X}_1^{\trans}P^{\trans} = \hat{F}_1P^{\trans}.
\end{align*}
Finally,
$||\hat{F}_1||_F^2 = \text{tr}(\hat{F}_1^{\trans} \hat{F}_1)
\overset{(\ref{prop:trace})}{=} \text{tr}(\hat{F}_2^{\trans}
\hat{F}_2) = ||\hat{F}_2||_F^2 $.
\end{proof}

Recall from Remark \ref{rem:partial_assgn} that Theorem \ref{thm:opt_feedback} is also valid for MGEAP with partial spectrum assignment. Next, we consider the case when only one real eigenvalue needs to assigned for the MGEAP while the remaining eigenvalues are unspecified. In this special case, we can explicitly characterize the global minimum of \eqref{eq:opt_cost} as shown in the next result. 

\begin{corollary} (\textbf{One real eigenvalue assignment}) \label{cor:1R_eig_assgn}
Let $\F = 1_{m\times n}$, $\Lambda \in \mathbb{R} $,  and $B = I_n$. Then, the global minima of the optimization problem \eqref{eq:opt_cost} is given by $\hat{F}_{gl} = -\sigma_{\text{min}}(A-\Lambda I_n)u v^{\trans}$, where $u$ and $v$ are unit norm left and right singular vectors, respectively, corresponding to $\sigma_{\text{min}}(A-\Lambda I_n)$. Further, $\Vert \hat{F}_{gl} \rVert_F = \sigma_{\text{min}}(A-\Lambda I_n)$.
 \end{corollary}
 
\begin{proof}
Since $\Lambda\in \mathbb{R}$, $\hat{X}\in\mathbb{R}^{n}\triangleq \hat{x}$ with $\|\hat{x}\|_2=1$, and $\hat{L}\in\mathbb{R}^{n}\triangleq \hat{l}$. Let $\hat{l} = \beta \hat{\tilde{l}}$  where $\beta \triangleq \lVert \hat{l} \rVert_2>0$.
Substituting $\hat{F} = - \hat{l} \hat{x}^{\trans}$ from \eqref{eq:opt_cond1} into \eqref{eq:right_eigvec}-\eqref{eq:left_eigvec}, we get
\begin{align*}
(A-\hat{l} \hat{x}^{\trans})\hat{x} &= \hat{x} \Lambda \Rightarrow (A-\Lambda I_n) \hat{x} = \beta \hat{\tilde{l}} \quad \text{and,}\\
(A^{\trans}-\hat{x} \hat{l}^{\trans} )\hat{l} &= \hat{l} \Lambda \Rightarrow (A-\Lambda I_n)^{\trans} \hat{\tilde{l}} = \beta\hat{x}.
\end{align*}
The above two equations imply that the unit norm vectors $\hat{x}$ and $\hat{\tilde{l}}$ are left and right singular vectors of $A-\Lambda I_n$ associated with the singular value $\beta$. Since $\lVert \hat{F} \rVert_F^2 = \text{tr}(\hat{F}^{\trans} \hat{F}) = \text{tr}(\hat{x}\hat{l}^{\trans} \hat{l}\hat{x}^{\trans}) = \beta^2$, we pick $\beta$ as the minimum singular value of $A-\Lambda I_n$, and the proof is complete. 
\end{proof}

We conclude this subsection by presenting a brief comparison of the non-sparse MGEAP solution with deflation techniques for eigenvalue assignment.
 For $B=I_n$, an alternative method to solve the non-sparse EAP is via the Wielandt deflation technique \cite{JHW:88}. Wielandt deflation achieves pole assignment by modifying the matrix $A$ in $n$ steps $A\rightarrow A_1\rightarrow A_2 \rightarrow \cdots \rightarrow A_n$. Step $i$ shifts one eigenvalue of $A_{i-1}$ to a desired location $\lambda_i$, while keeping the remaining eigenvalues of $A_{i-1}$ fixed. This is achieved by using the feedback $F^{i}_{df} = - (\mu_i - \lambda_i)v_iz_i^{\trans}$, where $\mu_i$ and $v_i$ are any eigenvalue and right eigenvector pair of $A_{i-1}$, and $z_i$ is any vector such that $z_i^{\trans}v_i = 1$. Thus, the overall feedback that solves the EAP is given as $F_{df} = \sum_{i=1}^n F^{i}_{df}$.

 It is interesting to compare the optimal feedback expression in
 \eqref{eq:opt_cond1},
 $\hat{F} = -\sum_{i=1}^n \hat{l}_i\hat{x}_i^{\trans}$, with the
 deflation feedback. Both feedbacks are sum of $n$ matrices, where
 each matrix has rank $1$. However, the Wielandt deflation has an
 inherent special structure and a restrictive property that, in each
 step, all except one eigenvalue remain unchanged.  Furthermore, each
 rank$-1$ term in $\hat{F}$ and $F_{df}$ involves the right/left
 eigenvectors of the closed and open loop matrix,
 respectively. Clearly, since $\hat{F}$ is the minimum-gain solution
 of \eqref{eq:opt_cost}, $||\hat{F}||_F\leq ||F_{df}||_F$.

\section{Solution algorithms}  \label{sec:algorithms}
In this section, we present an iterative algorithm to obtain a solution to the sparse MGEAP in \eqref{eq:opt_cost}. To develop the algorithm, we first present two algorithms for computing non-sparse and approximately-sparse solutions to the MGEAP, respectively. Next, we present two heuristic algorithms to obtain a sparse solution of the EAP (i.e. any sparse solution, which is not necessarily minimum-gain). Finally, we use these algorithms to develop the algorithm for sparse MGEAP. Note that although our focus is to develop the sparse MGEAP algorithm, the other algorithms presented in this section are novel in themselves to the best of our knowledge. 

We make the following assumptions:

\begin{assumption} \label{assump:no_fixed_modes}
The triplet $(A,B,\F)$ has no fixed modes, i.e., $\Gamma_{f}(A,B,\F) = \emptyset$. 
\end{assumption}

\begin{assumption} \label{assump:open_closed_eigv_disjoint}
The open and closed loop eigenvalue sets are
disjoint, i.e., $\Gamma(A)\cap \Gamma(\Lambda) =\emptyset$, and $B$
has full column rank.
\end{assumption}

Assumption \ref{assump:open_closed_eigv_disjoint} is not restrictive since if there are any
common eigenvalues in $A$ and $\Lambda$, we can use a preliminary
sparse feedback $F_p$ to shift the eigenvalues of $A$ to some other
locations such that $\Gamma(A+BF_p)\cap \Gamma(\Lambda)
=\emptyset$. Due to Assumption \ref{assump:no_fixed_modes}, such a $F_p$ always
exists. Then, we can solve the modified MGEAP\footnote{Although the
  minimization cost of the modified MGEAP is $0.5||F_p+F||_F^2$, it
  can be solved using techniques similar to solving MGEAP in
  \eqref{eq:opt_cost}.} with parameters $(A+BF_p, B,\Lambda,\F)$. If
$F$ is the sparse solution of this modified problem, then the solution
of the original problem is $F_p+F$.

To avoid complex domain calculations in the algorithms, we use the
real eigenvalue assignment constraint \eqref{eq:eigv_assgn_real}. For
convenience, we use a slight abuse of notation to denote $X_R$ and
$\Lambda_R$ as $X$ and $\Lambda$, respectively, in this section. Note
that the invertibility of $X$ is equivalent to the invertibility of
$X_R$.

\subsection{Algorithms for the non-sparse MGEAP}
We now present two iterative algorithms to obtain non-sparse and
approximately-sparse solutions to the MGEAP, respectively. To develop the
algorithms, we use the Sylvester equation based parametrization
\cite{SPB-EDS:82, AV:00a}. In this parametrization, instead of
defining $(F,X)$ as free variables, we define a parameter
$G\triangleq FX \in \mathbb{R}^{m \times n}$ as the free
variable. With this parametrization, the non-sparse MGEAP is stated as:

\begin{align} \label{eq:opt_cost_unsp2}
\underset{G}{\min} &\quad J =  \frac{1}{2}\:||F||_F^2 \hspace{40pt}
\end{align}
\ajustspaceandequationnumber
\begin{subequations}
\begin{align}  \label{eq:sylves}
\text{s.t.}\quad AX-X\Lambda+BG&=0, \\ \label{eq:F_G}
F = GX^{-1}&.
\end{align}
\end{subequations}
Note that, for any given $G$, we can solve the Sylvester equation \eqref{eq:sylves} to obtain $X$. Assumption \ref{assump:open_closed_eigv_disjoint} guarantees that \eqref{eq:sylves} has a unique solution \cite{NJH:02}. Further, we can use \eqref{eq:F_G} to obtain a non-sparse feedback matrix $F$. Thus, \eqref{eq:opt_cost_unsp2} is an unconstrained optimization problem in the free parameter $G$. 

The Sylvester-based parametrization requires the unique solution $X$
of \eqref{eq:sylves} to be non-singular, which holds generically if
(i) $(A,-BG)$ is controllable and (ii) $(\Lambda,-BG)$ is observable
\cite{SPB-EDS:81}. Since the system has no fixed modes, $(A,B)$ is
controllable \cite{SHW-EJD:73}. This implies that condition (i) holds
generically for almost all $G$. Further, since $B$ is of full column
rank, condition (ii) is guaranteed if $(\Lambda,-G)$ is
observable. These conditions are mild and are satisfied for almost all
instances as confirmed in our simulations (see Section
\ref{sec:num_study}).


The next result provides the gradient and Hessian of the cost $J$ w.r.t. to the parameter $g\triangleq \text{vec}(G)$.
\begin{lemma} \textbf{(Gradient and Hessian of $\mathbf{J})$} \label{lem:unsp_grad_Hess}
The gradient  and Hessian of the cost $J$ in \eqref{eq:opt_cost_unsp2} with respect to $g$ is given by
\begin{align} \label{eq:grad_unsp}
\frac{dJ}{dg}\!=&\!\underbrace{\left[(X^{-1}\!\otimes I_m)\!+\!(I_n\!\otimes B^{\trans})\tilde{A}^{-\trans}(X^{-1}\!\otimes F^{\trans})\right]}_{\textstyle \triangleq Z(F,X)} f, \\ \label{eq:Hess_unsp}
\frac{d^{2}J}{d^{2}g} &\triangleq H(F,X) = Z(F,X)Z^{\trans}\!(F,X) \nonumber\\
& + Z_1(F,X)Z^{\trans}(F,X) +  Z(F,X)Z_1^{\trans}(F,X),\\
\text{where} &\:\: Z_1(F,X)\triangleq (I_n\otimes B^{\trans}) \tilde{A}^{-\trans} (X^{-1}F^{\trans}\otimes I_n) T_{m,n},\nonumber \\
\text{and} &\:\: \tilde{A} \triangleq A\oplus (-\Lambda^{\trans}).\nonumber
\end{align}
\end{lemma}
 \begin{proof}
   Vectorizing \eqref{eq:sylves} using \ref{prop:vec1} and taking the
   differential,
 \begin{align}
 \tilde{A}x &+ (I_n\otimes B)g = 0 \nonumber \\ \label{eq:dx_dg}
 \Rightarrow dx &= -\tilde{A}^{-1} (I_n\otimes B)dg.
 \end{align}
 Note that due to Assumption \ref{assump:open_closed_eigv_disjoint}, $\tilde{A}$ is invertible. Taking the differential of \eqref{eq:F_G} and vectorizing, we get
 \begin{align}
 &\quad \quad dF \overset{\ref{prop:diff_inv}}{=} dG X^{-1} - \underbrace{G X^{-1}}_F dX X^{-1} \label{eq:dF_dG_dX}  \\
 &\overset{\ref{prop:vec1},\ref{prop:vec2}}{\Rightarrow} df = (X^{-\trans}\otimes I_m)dg - (X^{-\trans}\otimes F)dx \nonumber \\ 
& \overset{\eqref{eq:dx_dg}}{=} \underbrace{[(X^{-\trans}\otimes I_m) + (X^{-\trans}\otimes F)\tilde{A}^{-1} (I_n\otimes B)]}_{\textstyle  \overset{\ref{prop:kron}}{=}Z^{\trans}(F,X) }dg. \label{eq:df_dg}
 \end{align}
The differential of cost $J \overset{\ref{prop:frob}}{=} \frac{1}{2}f^{\trans}f$ is given by $dJ {=} f^{\trans}df$. Using \eqref{eq:df_dg} and \ref{prop:grad_Hess}, we get \eqref{eq:grad_unsp}. 
To derive the Hessian, we compute the second-order differentials of the involved variables. Note that since $g(G)$ is an independent variable, $d^{2}g = 0 (d^2G=0)$ \cite{JRM-HN:99}. Further, the second-order differential of \eqref{eq:sylves} yields $d^{2}X = 0$. Rewriting \eqref{eq:dF_dG_dX} as $dF X = dG-FdX$, and taking its differential and vectorization, we get
\begin{align}
(d^2F)X + (dF) (dX) &= -(dF) (dX)\nonumber \\
\Rightarrow d^2F &= -2(dF)(dX)X^{-1} \nonumber \\
\overset{\ref{prop:vec2}}{\Rightarrow} d^2f &= -2(X^{-\trans}\otimes dF)dx. \label{eq:d2f}
\end{align}


Taking the second-order differential of $J$ we get
\begin{align}
d^{2}J &= (df)^{\trans}df + f^{\trans} (d^{2}f) = (df)^{\trans}df + (d^{2}f)^{\trans} f \nonumber \\
 \overset{\eqref{eq:df_dg},\eqref{eq:d2f},\ref{prop:kron}}{=}& dg^{\trans} ZZ^{\trans} dg -2 dx^{\trans}\underbrace{(X^{-1}\otimes (dF)^{\trans})f}_{\overset{\ref{prop:kron}}{=}\text{vec}((dF)^{\trans} FX^{-\trans})} \nonumber \\
   \overset{\ref{prop:kron}}{=} &dg^{\trans} ZZ^{\trans} dg - 2dx^{\trans} (X^{-1}F^{\trans}\otimes I_n) T_{m,n}df\nonumber\\
  \overset{\eqref{eq:dx_dg},\eqref{eq:df_dg}}{=} &dg^{\trans} ZZ^{\trans} dg \nonumber\\
 +& \underbrace{2dg^{\trans} (I_n\otimes B^{\trans}) \tilde{A}^{-\trans} (X^{-1}F^{\trans}\otimes I_n) T_{m,n} Z^{\trans} dg}_{dg^{\trans}(Z_1 Z^{\trans} + Z Z_1^{\trans}) dg}. \label{eq:d2J}
\end{align}
The Hessian in \eqref{eq:Hess_unsp} follows from \eqref{eq:d2J} and
\ref{prop:grad_Hess}.
\end{proof}
 
The first-order optimality condition of the unconstrained problem
\eqref{eq:opt_cost_unsp2} is $\frac{dJ}{dg} = Z(F,X)f=0$. The next
result shows that this condition is equivalent to the first-order
optimality conditions of Theorem \ref{thm:opt_feedback} without
sparsity constraints.

\begin{corollary}\textbf{(Equivalence of first-order optimality conditions)}
  Let $\F = 1_{m\times n}$. Then, the first-order optimality condition
  $Z(\hat{F},\hat{X})\hat{f} = 0$ of \eqref{eq:opt_cost_unsp2} is
  equivalent to \eqref{eq:opt_cond1}-\eqref{eq:left_eigvec}, where
\begin{align} \label{eq:l_unsp}
\hat{l}\triangleq \text{vec}(\hat{L}) = \tilde{A}^{-\trans}(\hat{X}^{-1}\!\otimes \hat{F}^{\trans})\hat{f}.
\end{align}
\end{corollary}
\begin{proof}
The optimality condition \eqref{eq:right_eigvec} follows from \eqref{eq:sylves}-\eqref{eq:F_G}. Equation \eqref{eq:opt_cond1} can be rewritten as $\hat{F}\hat{X}^{-\trans} + B^{\trans}\hat{L}=0$ and its vectorization using \ref{prop:vec1} yields $Z(\hat{F},\hat{X})\hat{f} = 0$. Finally, vectorization of the left side of \eqref{eq:left_eigvec} yields
\begin{align*}
  \text{vec}[&(A+B\hat{F})^{\trans}\hat{L}-\hat{L}\Lambda^{\trans}] = \text{vec}[A^{\trans}\hat{L}-\hat{L}\Lambda^{\trans} + (B\hat{F})^{\trans}L] \\
             &\overset{\ref{prop:vec1},\ref{prop:kron}}{=} \tilde{A}^{\trans}\hat{l} + (I_n\otimes (B\hat{F})^{\trans})\hat{l} \\
             &\overset{\eqref{eq:l_unsp}}{=} (\hat{X}^{-1}\!\otimes \hat{F}^{\trans})\hat{f} + (I_n\otimes (B\hat{F})^{\trans})\tilde{A}^{-\trans}(\hat{X}^{-1}\!\otimes \hat{F}^{\trans})\hat{f} \\
             &\overset{\ref{prop:kron}}{=} (I_n\otimes \hat{F}^{\trans}) Z(\hat{F},\hat{X})\hat{f} = 0.
\end{align*}
To conclude, note that $\hat{L}$ is the right eigenvector matrix.
\end{proof}

Using Lemma \ref{lem:unsp_grad_Hess}, we next present a steepest/Newton descent algorithm to solve the non-sparse MGEAP \eqref{eq:opt_cost_unsp2} \cite{DGL-YY:08}. In the algorithms presented in  this section, we interchangeably use the the matrices $(G,F,X)$ and their respective vectorizations $(g,f,x)$. The conversion of a matrix to the vector (and vice-versa) is not specifically stated in the steps of the algorithms and is assumed wherever necessary. 
\begin{algorithm} \label{algo:unsp_des}
  \KwIn{$A,B,\Lambda,G_0.$}
  \KwOut{Local minimum $(\hat{F},\hat{X})$ of \eqref{eq:opt_cost_unsp2}.}
  \BlankLine
  \textbf{Initialize:} $G_0,X_0\leftarrow$ Solution of \eqref{eq:sylves}, $F_0  \leftarrow G_0X_0^{-1}$ \\
 \Repeat{\textup{convergence}}
 {\nl $\alpha \leftarrow$ Compute step size (see below)\label{line:step_size}\;
 \nl  $g \leftarrow  g - \alpha Z(F,X)f\quad$ \textbf{or}\label{line:grad_des}\;
 \nl $g \leftarrow  g - \alpha [H(F,X)+V(F,X)]^{-1}Z(F,X)f$\label{line:newton_des}\;
 \vspace{4pt}
 $X \leftarrow$ Solution of Sylvester equation \eqref{eq:sylves}\;
 $F \leftarrow GX^{-1}$\;
 }
 \Return{$(F,X)$}
  \caption{Non-sparse solution to the MGEAP}
\end{algorithm}

Steps \ref{line:grad_des} and $\ref{line:newton_des}$ of Algorithm
\ref{algo:unsp_des} represent the steepest and (damped) Newton descent
steps, respectively. Since, in general, the Hessian $H(F,X)$ is not
positive-definite, the Newton descent step may not result in a
decrease of the cost. Therefore, we add a Hermitian matrix $V(F,X)$ to
the Hessian to make it positive definite \cite{DGL-YY:08}. We will
comment on the choice of $V(F,X)$ in Section \ref{sec:num_study}. In
step \ref{line:step_size}, the step size $\alpha$ can be determined by
backtracking line search or Armijo's rule \cite{DGL-YY:08}. For a
detailed discussion of the steepest/Newton descent methods, the reader
is referred to \cite{DGL-YY:08}. The computationally intensive steps
in Algorithm \ref{algo:unsp_des} are solving the Sylvester equation
\eqref{eq:sylves} and evaluating the inverses of $X$ and $H+V$. Note
that the expression of the gradient in \eqref{eq:grad_unsp} is similar
to the expression provided in \cite{AV:00a}. However, the expression
of Hessian in \eqref{eq:Hess_unsp} is new and it allows us to
implement Newton descent whose convergence is considerably faster than
steepest descent.


Next, we present a relaxation of the optimization problem \eqref{eq:opt_cost} and a corresponding algorithm that provides approximately-sparse solutions to the MGEAP. We remove the explicit feedback sparsity constraints \eqref{eq:spar_const} and modify the cost function to penalize it when these sparsity constraints are violated. Using the Sylvester equation based parametrization, the relaxed optimization problem is stated as: 
\begin{align} \label{eq:opt_cost_pen}
\underset{G}{\min} &\quad J_{W} =  \frac{1}{2}\:||W\circ F||_F^2  \hspace{10pt}
\end{align}
\ajustspaceandequationnumber
\begin{subequations}
\begin{align*} 
\text{s.t.}\quad \eqref{eq:sylves} \:\: \text{and} \:\: \eqref{eq:F_G} \:\: \text{hold true,}
\end{align*}
\end{subequations}
where $W\in\mathbb{R}^{m\times n}$ is a weighing matrix that penalizes the cost for violation of sparsity constraints, and is given by 
\begin{align*}
W_{ij} = 
\begin{cases}
1  \quad &\text{if} \:\:\  \F_{ij} = 1, \:\text{and} \\
\gg 1 \quad &\text{if} \:\:\  \F_{ij} = 0.
\end{cases}
\end{align*}
As the penalty weights of $W$ corresponding to the sparse entries of $F$ increase, an optimal solution of \eqref{eq:opt_cost_pen} becomes more sparse and approaches towards the optimal solution of $\eqref{eq:opt_cost}$. Note that the relaxed problem \eqref{eq:opt_cost_pen} corresponds closely to the non-sparse MGEAP \eqref{eq:opt_cost_unsp2}. Thus, we use a similar gradient based approach to obtain its solution. 
\begin{lemma} \textbf{(Gradient and Hessian of $\mathbf{J_W}$)} \label{lem:grad_hess_approx}
The gradient  and Hessian of the cost $J_W$ in \eqref{eq:opt_cost_pen} with respect to $g$ is given by
\begin{align} \label{eq:grad_pen}
\frac{dJ_{W}}{dg}&=Z(F,X) \bar{W}f , \\ 
\frac{d^{2}J_W}{d^{2}g} &\triangleq H_W(F,X) = Z(F,X)\bar{W}Z^{\trans}\!(F,X) \nonumber\\
 + Z_{1,W}(F,X)&Z^{\trans}(F,X) +  Z(F,X)Z_{1,W}^{\trans}(F,X),\label{eq:Hess_pen} \\
\text{where} \quad \bar{W} &\triangleq \text{diag}(\text{vec}(W\circ W)) \quad \text{and,} \nonumber \\
Z_{1,W}(F,X)\triangleq &(I_n\!\otimes\! B^{\trans}) \tilde{A}^{-\trans} (X^{-1}(W\!\circ\! W \!\circ\! F)^{\trans}\!\otimes\! I_n) T_{m,n}.\nonumber 
\end{align}
\end{lemma}
\begin{proof}
  Since the constraints of problems \eqref{eq:opt_cost_unsp2} and
  \eqref{eq:opt_cost_pen} coincide, Equations
  \eqref{eq:dx_dg}-\eqref{eq:d2f} from Lemma \ref{lem:unsp_grad_Hess}
  also hold true for problem \eqref{eq:opt_cost_pen}. Now,
  $J_{W} \overset{\ref{prop:frob},\ref{prop:had2}}{=} \frac{1}{2}
  (\text{vec}(W)\circ f)^{\trans}(\text{vec}(W)\circ f) =\frac{1}{2}
  f^{\trans}\bar{W}f$. Thus, $dJ_{W} = f^{\trans}\bar{W}df$ and
  $d^{2}J_{W} = (df)^{\trans}\bar{W}df +
  f^{\trans}\bar{W}d^{2}f$. Using the relation
  $\text{vec}(W\circ W \circ F) = \bar{W}f$, the remainder of the
  proof is similar to proof of Lemma \ref{lem:unsp_grad_Hess}.
\end{proof}
Using Lemma \ref{lem:grad_hess_approx}, we next present an algorithm to obtain an approximately-sparse solution to the MGEAP.

\begin{algorithm} \label{algo:unsp_des_pen}
  \KwIn{$A,B,\Lambda,W, G_0$}
  \KwOut{Local minimum $(\hat{F},\hat{X})$ of \eqref{eq:opt_cost_pen}.}
  \BlankLine
  \textbf{Initialize:} $G_0,X_0\leftarrow$ Solution of \eqref{eq:sylves}, $F_0  \leftarrow G_0X_0^{-1}$\\
 \Repeat{\textup{convergence}}
 {\nl $\alpha\leftarrow$ Update step size\;
 \nl  $g \leftarrow  g - \alpha Z(F,X)\bar{W}f\quad$ \textbf{or}\label{line:grad_des3}\; 
 \nl $g \leftarrow  g - \alpha [H_{W}(F,X)+V_{W}(F,X)]^{-1}Z(F,X)\bar{W}f$\label{line:newton_des3}\;
 \vspace{4pt}
 $X \leftarrow$ Solution of Sylvester equation \eqref{eq:sylves}\;
 $F \leftarrow GX^{-1}$\;
 }
 \Return{$(F,X)$}
  \caption{Approximately-sparse solution to the \hspace*{65pt} MGEAP}
\end{algorithm}
The step size rule and modification of the Hessian in Algorithm \ref{algo:unsp_des_pen} is similar to Algorithm \ref{algo:unsp_des}.

\subsection{Algorithms for the sparse EAP}
In this subsection, we present two heuristic algorithms to obtain a
sparse solution to the EAP (not necessarily minimum-gain). This
involves finding a pair $(F,X)$ which satisfies the eigenvalue
assignment and sparsity constraints \eqref{eq:eigv_assgn},
\eqref{eq:spar_const}. We begin with a result that combines these two
constraints.

\begin{lemma} \textbf{(Feasibility of $\mathbf{(F,X)}$)}\label{lem:fx_exist_cond}
An invertible matrix $X\in\mathbb{R}^{n \times n}$ satisfies \eqref{eq:eigv_assgn} and \eqref{eq:spar_const} if and only if 
\begin{align}  \label{eq:spEAP_cond1}
\tilde{a}(x) &\in \mc{R}(\tilde{B}(X)) \quad \text{where,}\\ \nonumber
 \tilde{a}(x) \triangleq \tilde{A}x, \tilde{B}(X) &\triangleq  - (X^{\trans}\otimes B)P_{\F}, P_{\F} \triangleq \text{diag}(\text{vec}(\F)).
 \end{align}
Further, if \eqref{eq:spEAP_cond1} holds true, then the set of sparse feedback matrices that satisfy \eqref{eq:eigv_assgn} and \eqref{eq:spar_const} is given by 
\begin{align} \label{eq:sp_feedback_set}
\mc{F}_{X} =  \{P_{\F} f_{ns}: \tilde{B}(X)f_{ns} = \tilde{a}(x), f_{ns}\in\real^{mn}\}.
 \end{align}
\end{lemma}
\begin{proof}
Any feedback $f$ which satisfies the sparsity constraint \eqref{eq:spar_const} can be written as $f = P_{\F}f_{ns}$ where $f_{ns}\in\mathbb{R}^{mn}$ is a non-sparse vector\footnote{Since $f$ satisfies \eqref{eq:spar_const_vec}, it can also be characterized as $f = (I_{mn}-Q^{+}Q)f_{ns}$, and thus $P_{\F} =I_{mn}-Q^{+}Q$.}. Vectorizing \eqref{eq:eigv_assgn} using \ref{prop:vec1} and \ref{prop:vec2}, and substituting $f = P_{\F}f_{ns}$, we get
 \begin{align}  \label{eq:spEAP_cond2}
\tilde{A}x = - (X^{\trans}\otimes B)P_{\F}f_{ns},
\end{align}
from which \eqref{eq:spEAP_cond1} and \eqref{eq:sp_feedback_set} follow.
\end{proof}

Based on Lemma \ref{lem:fx_exist_cond}, we develop a heuristic
algorithm for a sparse solution to the EAP. The algorithm starts with
a non-sparse EAP solution $(F_0,X_0)$ that does not satisfy
\eqref{eq:spar_const} and \eqref{eq:spEAP_cond1}. Then, it takes
repeated projections of $\tilde{a}(x)$ on $\mathcal{R}(\tilde{B}(X))$
to update $X$ and $F$, until a sparse solution is obtained.

\begin{algorithm} \label{algo:sp_EAP1}
  \KwIn{$A,B,\Lambda,\F,G_0,iter_{max}.$}
  \KwOut{$(F,X)$ satisfying \eqref{eq:eigv_assgn} and \eqref{eq:spar_const}.}
  \BlankLine
  \textbf{Initialize:} $G_0$, $X_0\!\!\leftarrow$ Solution of \eqref{eq:sylves}, $F_{ns,0} \leftarrow G_0X_0^{-1}$, $i\leftarrow 0$ \\
 \Repeat{\textup{convergence or} $i > iter_{max}$}
 { \nl $\tilde{a}(x) \leftarrow \tilde{B}(X)[\tilde{B}(X)]^{+}\tilde{a}(x)$\label{line:EAP1_1}\;
   \vspace{2pt}
   \nl $x \leftarrow  \tilde{A}^{-1} \tilde{a}(x) $\label{line:EAP1_2}\;
   \nl $X \leftarrow$ Normalize $X$\label{line:EAP1_normalize}\;
   $i \leftarrow i + 1\;$
 }
 \Return{$(f \in \mc{F}_X \textup{ in } \eqref{eq:sp_feedback_set},X)$}
  \caption{Sparse solution to EAP}
\end{algorithm}
In step \ref{line:EAP1_1} of Algorithm \ref{algo:sp_EAP1}, we update
$\tilde{a}(x)$ by projecting it on $\mathcal{R}(\tilde{B}(X))$. Step
\ref{line:EAP1_2} computes $x$ from $\tilde{a}(x)$ using the fact that
$\tilde{A}$ is invertible (c.f. Assumption \ref{assump:open_closed_eigv_disjoint}). Finally, the
normalization in step \ref{line:EAP1_normalize} is performed to ensure
invertibility of $X$\footnote{Since $X$ is not an eigenvector matrix,
  we compute the eigenvectors from $X$, normalize them,
  and then recompute real $X$.}.

Next, we develop a second heuristic algorithm for solving the sparse EAP problem using the non-sparse MGEAP solution in Algorithm \ref{algo:unsp_des}. The algorithm starts with a non-sparse EAP solution $(F_0,X_0)$. In each iteration, it sparsifies the feedback to obtain $f = P_{\F}f_{ns}$ (or $F=\F\circ F_{ns}$), and then solves the following non-sparse MGEAP  
\begin{align} \label{eq:opt_cost_unsp1}
\underset{F_{ns},X}{\min} &\quad \frac{1}{2}\:||F_{ns}-F||_F^2 \hspace{27pt}
\end{align}
\ajustspaceandequationnumber
\begin{subequations}
\begin{align}  \label{eq:eigv_assgn_unsp1}
\text{s.t.}& \quad (A+BF_{ns})X=X\Lambda,
\end{align}
\end{subequations}
to update $F_{ns}$ that is \emph{close} to the sparse $F$. This
algorithm resembles to the alternating projection method
\cite{RE-MR:11} to find an intersection point of two sets. The
operation $F=\F\circ F_{ns}$ computes the projection of $F_{ns}$ on
the convex set of sparse feedback matrices. The optimization problem
\eqref{eq:opt_cost_unsp1} computes the projection of $F$ on the
non-convex set of feedback matrices that assign the eigenvalues.
Thus, using the heuristics of repeated sparsification of the solution
of non-sparse MGEAP in \eqref{eq:opt_cost_unsp1}, the algorithm
obtains a sparse solution. The alternating projection method is not
guaranteed to converge in general when the sets are not
convex. However, if the starting point is \emph{close} to the two
sets, convergence is guaranteed \cite{SAL-JM:08}.

Note that a solution $\hat{F}_{ns}$ of the problem
\eqref{eq:opt_cost_unsp1} with parameters $(A,B,\Lambda,F)$ satisfies
$\hat{F}_{ns} = F + \hat{K}_{ns}$, where $\hat{K}_{ns}$ is a solution
of the optimization problem \eqref{eq:opt_cost_unsp2} with parameters
$(A+BF,B,\Lambda)$. Thus, we can use Algorithm \ref{algo:unsp_des} to
solve \eqref{eq:opt_cost_unsp1}.

\begin{algorithm} \label{algo:sp_EAP2}
  \KwIn{$A,B,\Lambda,\F,G_0, iter_{max}.$}
  \KwOut{$(F,X)$ satisfying \eqref{eq:eigv_assgn} and \eqref{eq:spar_const}.}
  \BlankLine
 \textbf{Initialize:} $G_0$, $X_0\leftarrow$ Solution of \eqref{eq:sylves}, $F_{ns,0}  \leftarrow G_0X_0^{-1}$,  $i\leftarrow 0$ \\
 \Repeat{\textup{convergence or} $i >iter_{max}$}
 { $F \leftarrow\F\circ F_{ns}$\;
   \nl $(K_{ns},X) \leftarrow \text{Algorithm \ref{algo:unsp_des}} (A+BF,B,\Lambda)$\label{line:sp_EAP2_2}\;
   $F_{ns}\leftarrow F+K_{ns}$\;
   $i \leftarrow i + 1$\;
 }
 \Return{$(F,X)$}
  \caption{Projection-based sparse solution to EAP}
\end{algorithm}

\begin{remark} \textbf{(Comparison of EAP Algorithms \ref{algo:sp_EAP1} and \ref{algo:sp_EAP2})} \label{rem:good_proj}

\noindent \emph{1. Projection property:} In general, Algorithm \ref{algo:sp_EAP1} results in a sparse EAP solution $F$ that is considerably different from the initial non-sparse $F_{ns,0}$. In contrast, Algorithm \ref{algo:sp_EAP2} provides a sparse solution $F$ that is \emph{close} to $F_{ns,0}$. This is due to the fact that Algorithm \ref{algo:sp_EAP2} updates the feedback by solving the optimization problem \eqref{eq:opt_cost_unsp1}, which minimizes the deviations between successive feedback matrices. Thus, Algorithm \ref{algo:sp_EAP2} provides a good (although not necessarily orthogonal) projection of a given non-sparse EAP solution $F_{ns,0}$ on the space of sparse EAP solutions. 

\noindent \emph{2. Complexity:} The computational complexity of Algorithm \ref{algo:sp_EAP2} is considerably larger than that of Algorithm \ref{algo:sp_EAP1}. This is because Algorithm \ref{algo:sp_EAP2} requires a solution of a non-sparse MGEAP problem in each iteration. In contrast,  Algorithm \ref{algo:sp_EAP1} only requires projections on the range space of a matrix in each iteration. Thus, Algorithm \ref{algo:sp_EAP1} is considerably faster as compared to Algorithm \ref{algo:sp_EAP2}.

\noindent \emph{3. Convergence:}  Although we do not formally prove the convergence of heuristic Algorithms \ref{algo:sp_EAP1} and \ref{algo:sp_EAP2} in this paper, a comprehensive simulation study in Subsection \ref{subsec:Num_study} suggests that Algorithm \ref{algo:sp_EAP2}  converges in almost all instances. In contrast, Algorithm \ref{algo:sp_EAP1} converges in much fewer instances and its convergence deteriorates considerably as the number of sparsity constraints increase (see Subsection \ref{subsec:Num_study}). \oprocend
 \end{remark}  
 
 If the starting point $F_{ns,0}$ of Algorithm 4 is ``sufficiently
 close'' to a local minima $\hat{F}$ of \eqref{eq:opt_cost}, then its
 iterations will converge (heuristically) to $\hat{F}$. In this case,
 Algorithm 4 can be used to solve the sparse MGEAP. However,
 convergence to $\hat{F}$ is not guaranteed for an arbitrary starting
 point.

\subsection{Algorithm for sparse MGEAP}
In this subsection, we present an iterative projected gradient algorithm to compute the sparse solutions of the MGEAP in \eqref{eq:opt_cost}. The algorithm consists of two loops. The outer loop is same as the non-sparse MGEAP Algorithm \ref{algo:unsp_des} (using steepest descent) with an additional projection step, which constitutes the inner loop. Figure \ref{fig:single_iter} represents one iteration of the algorithm. First, the gradient $Z(F_k,X_k)f_k$ is computed at a current point $G_k$ (equivalently $(F_k,X_k)$, where $F_k$ is sparse). Next, the gradient is projected on the tangent plane of the sparsity constraints \eqref{eq:spar_const_vec}, which is given by
\begin{align}
\mc{T}_F &= \left\{y\in\mathbb{R}^{mn}: \left[\frac{d(Qf)}{dg}\right]^{\trans} y = 0\right\} \nonumber \\
&  = \left\{y\in\mathbb{R}^{mn}:QZ^{\trans}(F,X)y = 0\right\}.
\end{align}
From \ref{prop:projection}, the projection of the gradient on $\mc{T}_F$ is given by $P_{F_k}Z(F_k,X_k)f_k$, where $P_{F_k} = I_{mn} - [QZ^{\trans}(F_k,X_k)]^{+}[QZ^{\trans}(F_k,X_k)]$. Next, a move is made in the direction of the projected gradient to obtain $G_{ns,k} (F_{ns,k}, X_{ns,k})$. Finally, the orthogonal projection of $G_{ns,k}$ is taken on the space of sparsity constraints to obtain $G_{k+1} (F_{k+1},X_{k+1})$. This orthogonal projection is equivalent to solving \eqref{eq:opt_cost_unsp1} with sparsity constraints \eqref{eq:spar_const}, which in turn is equivalent to the original sparse MGEAP \eqref{eq:opt_cost}. Thus, the orthogonal projection step is as difficult as the original optimization problem. To address this issue, we use the heuristic Algorithm \ref{algo:sp_EAP2} to compute the projections. Although the projections obtained using Algorithm \ref{algo:sp_EAP2} are not necessarily orthogonal, they are typically good (c.f. Remark \ref{rem:good_proj}). 

\begin{figure}[h!]
\centering
\includegraphics[width=\columnwidth]{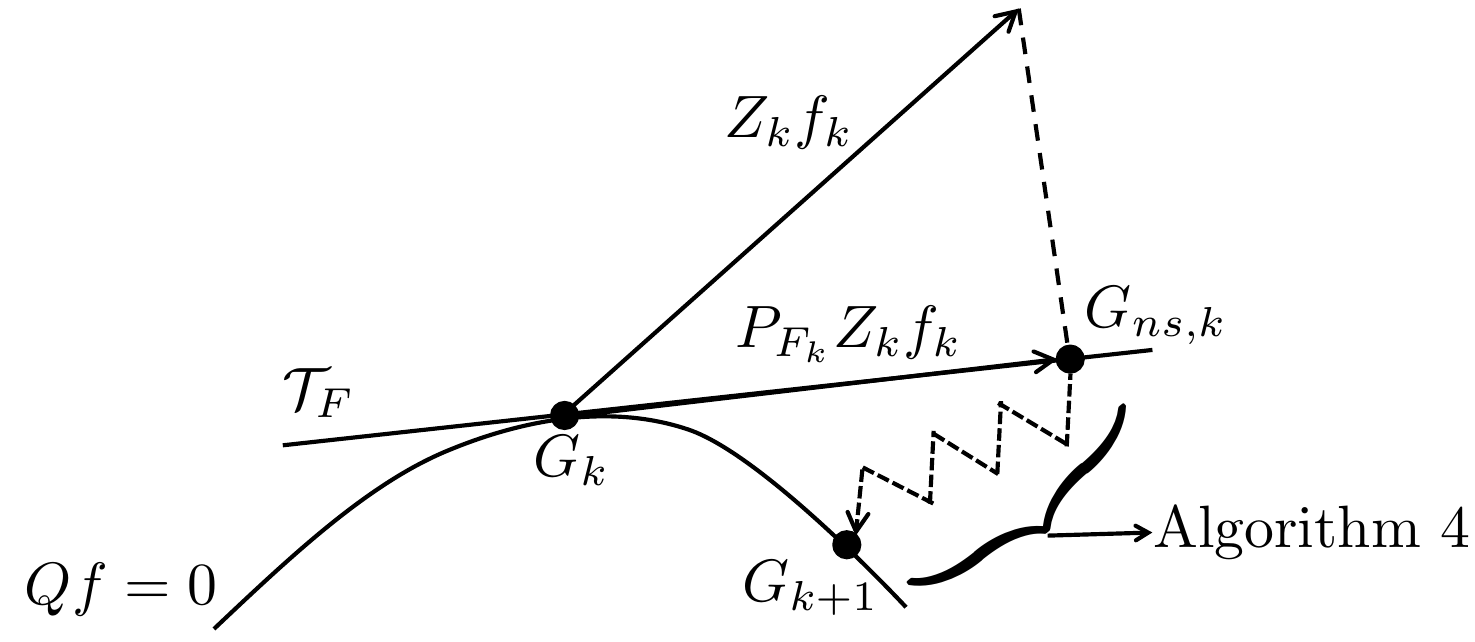}
\caption{A single iteration of Algorithm \ref{algo:sp_des}.}
\label{fig:single_iter}
\end{figure}

%

\begin{algorithm} \label{algo:sp_des}
  \KwIn{$A,B,\Lambda,\F,G_0, iter_{max}.$}
  \KwOut{Local minimum $(\hat{F},\hat{X})$ of \eqref{eq:opt_cost}.}
  \BlankLine
\textbf{Initialize:} $(F_0,X_0)\leftarrow \text{Algorithm } \ref{algo:sp_EAP2} (A,B,\Lambda,\F,G_0,iter_{max})$, $G_0 \leftarrow F_0X_0$, $i \leftarrow 0$  \\
 \Repeat{\textup{convergence or} $i > iter_{max}$}
{ $\alpha \leftarrow$ Update step size\; 
   $g_{ns} \leftarrow  g - \alpha P_F Z(F,X) f$\; 
 $X_{ns} \leftarrow$ Solution of Sylvester equation \eqref{eq:sylves}\;
 $F_{ns} \leftarrow G_{ns}X_{ns}^{-1}$\;
\nl $(F,X) \leftarrow \text{Algorithm \ref{algo:sp_EAP2}} (A,B,\Lambda,\F,G_{ns},iter_{max})$\label{line:sp_MGEAP1}\;
$G\leftarrow FX$\;
$i \leftarrow i + 1$\;
 }
 \Return{$(F,X)$}
  \caption{Sparse solution to the MGEAP}
\end{algorithm}
Algorithm \ref{algo:sp_des} is computationally intensive due to the use of Algorithm \ref{algo:sp_EAP2} in step \ref{line:sp_MGEAP1} to compute the projection on the space of sparse matrices. In fact, the computational complexity of Algorithm \ref{algo:sp_des} is one order higher than that of non-sparse MGEAP Algorithm \ref{algo:unsp_des}. However, a way to considerably reduce the number of iterations of Algorithm \ref{algo:sp_des} is to initialize it using the approximately-sparse solution obtained by Algorithm \ref{algo:unsp_des_pen}. In this case, Algorithm \ref{algo:sp_des} starts \emph{near} the the local minimum and, thus, its convergence time reduces considerably.

\section{Simulation studies} \label{sec:num_study} In this section, we
present the implementation details of the algorithms developed in
Section \ref{sec:algorithms} and provide numerical simulations to
illustrate their properties.

\subsection{Implementation aspects of the algorithms}
In the Newton descent step (step \ref{line:newton_des}) of Algorithm \ref{algo:unsp_des}, we need to choose (omitting the parameter dependence notation) $V$ such that $H + V$ is positive-definite. We choose $V= \delta I_{mn} - Z_1 Z^{\trans} - Z Z_1^{\trans}$ where, $0<\delta \ll 1$. Thus, from \eqref{eq:Hess_unsp}, we have: $H+V = ZZ^{\trans}+\epsilon I_{mn}$. Clearly, $ZZ^{\trans}$ is positive-semidefinite and the small additive term $\epsilon I_{mn}$ ensures that $H+V$ is positive-definite. Note that other possible choices of $V$ also exist. In step \ref{line:step_size} of Algorithm \ref{algo:unsp_des}, we use the Armijo rule to compute the step size $\alpha$. Finally, we use $\left\lVert\frac{dJ}{dg}\right\rVert_2 < \epsilon$, $0<\epsilon \ll 1$ as the convergence criteria of Algorithm \ref{algo:unsp_des}. For Algorithm \ref{algo:unsp_des_pen}, we analogously choose $V_{W} = \delta I_{mn} - Z_{1,W} Z^{\trans} - Z Z_{1,W}^{\trans}$ and the same convergence criteria and step size rule as Algorithm \ref{algo:unsp_des}. 
In both algorithms, if we encounter a scenario in which the solution $X$ of \eqref{eq:sylves} is singular (c.f. paragraph below \eqref{eq:F_G}), we perturb $G$ slightly such that the new solution is non-singular, and then continue the iterations. We remark that such instances occur extremely rarely in our simulations.


For the sparse EAP Algorithm \ref{algo:sp_EAP1}, we use the convergence criteria $e_X = \lVert [I_{n^2}- \tilde{B}(X)(\tilde{B}(X))^{+}]\tilde{a}(x) \rVert_2 < \epsilon$, $0<\epsilon \ll 1$. For Algorithm \ref{algo:sp_EAP2}, we use the convergence criteria $e_F = \lVert F-\F\circ F\rVert_F<\epsilon$, $0<\epsilon \ll 1$. Thus, the iterations of these algorithms stop when $x$ lies in a certain subspace and when the sparsity error becomes \emph{sufficiently} small (within the specified tolerance), respectively. Further, note that Algorithm \ref{algo:sp_EAP2} uses Algorithm \ref{algo:unsp_des} in step \ref{line:sp_EAP2_2} without specifying an initial condition $G_0$ for the latter. This is because in step \ref{line:sp_EAP2_2}, we effectively run Algorithm \ref{algo:unsp_des} for multiple initial conditions in order to capture its global minima. We remark that the capture of global minima by Algorithm \ref{algo:unsp_des} is crucial for convergence of Algorithm \ref{algo:sp_EAP2}.

As the iterations of Algorithm \ref{algo:sp_EAP2} progress, the sparse matrix $F$ achieves eigenvalue assignment with increasing accuracy. As a result, near the convergence of Algorithm \ref{algo:sp_EAP2}, the eigenvalues of $A+BF$ and $\Lambda$ are very close to each other. This creates numerical difficulties when Algorithm \ref{algo:unsp_des} is used with parameters $(A+BF,B,\Lambda)$ in step \ref{line:sp_EAP2_2} (see Assumption \ref{assump:open_closed_eigv_disjoint}). To avoid this issue, we run Algorithm \ref{algo:unsp_des} using a preliminary feedback $F_p$, as explained below Assumption \ref{assump:open_closed_eigv_disjoint}.

Finally, for Algorithm \ref{algo:sp_des}, we use the following
convergence criteria:
$\left\lVert P_F\frac{dJ}{dg}\right\rVert_2 < \epsilon$,
$0<\epsilon \ll 1$. We choose the stopping tolerance $\epsilon$
between $10^{-6}$ and $10^{-5}$ for all the algorithms.

\subsection{Numerical study}   \label{subsec:Num_study}
We begin this subsection with the following example:
\begin{align*}
\small A &=\begin{bmatrix}
-3.7653  & -2.1501 &   \phantom{-}0.3120 &  -0.2484 \\
    \phantom{-}1.6789  &  \phantom{-}1.0374  & -0.5306   & \phantom{-}1.3987\\
   -2.1829 &  -2.5142  & -1.2275  &  \phantom{-}0.2833 \\
  -13.6811  & -9.6804  & -0.5242   & \phantom{-}2.9554
\end{bmatrix}\!,\\ 
\small B &=  \begin{bmatrix}
1 & 1 & 2 & 5\\
1 & 3 & 4 & 2
\end{bmatrix}^{\trans}\!\!\!\!, \qquad
\F = \begin{bmatrix}
1 & 1 & 0 & 0\\
1 & 0 & 1 & 1
\end{bmatrix}\!,\\
\mc{S} &= \{-2,-1, -0.5\pm j\}.
\end{align*}

\begin{table}[h!]
\caption{Comparison of MGEAP solutions by Algorithms \ref{algo:unsp_des}, \ref{algo:unsp_des_pen} and \ref{algo:sp_des}}
\begin{center}
\setlength{\tabcolsep}{2pt}
 \begin{tabular}{| l |  }
  \hline	\\[-0.7em]		
  \textbf{Non-sparse solutions by Algorithm \ref{algo:unsp_des}} \\
  \hline  
  \rule{0pt}{21pt}  $\hat{X}_1 =
  \begin{bmatrix} 
  -0.0831 + 0.3288j & \multirow{4}{*}{$\hat{x}_1^{*}$} &  -0.5053 &   \phantom{-}0.2031  \\
   \phantom{-}0.1919 - 0.4635j  &  &  \phantom{-}0.5612 & -0.2505 \\
   \phantom{-}0.1603 + 0.5697j &    & \phantom{-}0.5617   & \phantom{-}0.8441\\
   \phantom{-}0.3546 + 0.3965j &  &  -0.3379  &  \phantom{-}0.4283 
   \end{bmatrix}$ \\ 
     \rule{0pt}{21pt}$\hat{L}_1 =\:\begin{bmatrix}
  -0.5569 + 1.4099j &   \multirow{4}{*}{\phantom{-}$\hat{l}_1^{*}$\phantom{-}} & -0.9154  &  \phantom{-}2.5568\\
  -0.2967 + 1.0244j &  & -0.5643   & \phantom{-}1.7468\\
  -0.0687 + 0.0928j &    & \phantom{-}0.0087    & \phantom{-}0.2722\\
   \phantom{-}0.2259 - 0.2523j &   & \phantom{-}0.0869   & -0.4692
   \end{bmatrix}$ \\
   \rule{0pt}{15pt}$\hat{F}_1 = \begin{bmatrix}
   -0.1111 &  -0.1089  & -0.0312  & -0.4399\\
   -0.1774 &  -0.2072  &  \phantom{-}0.0029  &  \phantom{-}0.1348
   \end{bmatrix}$, $\lVert\hat{F}_1\rVert_F = 0.5580$ \\
   \rule{0pt}{15pt}$\hat{F}_2 = \begin{bmatrix}
    \phantom{-}0.3817 &  -0.3349  &  \phantom{-}0.7280 &  -0.2109\\
   -0.0873  & -0.4488  & -0.4798 &  -0.0476
   \end{bmatrix}$, $\lVert\hat{F}_2\rVert_F = 1.1286$ \\
   \rule{0pt}{15pt} $\hat{F}_3 = \begin{bmatrix}
    0.3130  &  \phantom{-}2.0160 &   \phantom{-}1.2547 &  -0.6608\\
    0.0683 &  -0.7352 &  -0.0748 &  -1.0491
   \end{bmatrix}$, $\lVert\hat{F}_3\rVert_F = 2.7972$ \\
    \rule{0pt}{10pt} Average \# of Steepest/Newton descent iterations = $5402.1/15.5$\\
   \hline	\hline	\\[-0.7em]			
  \textbf{Approximately-sparse solutions by Algorithm \ref{algo:unsp_des_pen} with $\mathbf{w=30}$} \\
  \hline  
    \rule{0pt}{15pt}$\hat{F}_1 = \begin{bmatrix}
   0.9652 &  -1.3681  &  \phantom{-}0.0014  & -0.0021\\
    0.4350  & -0.0023 &  -0.6746 &  -0.1594
   \end{bmatrix}$, $\lVert\hat{F}_1\rVert_F = 1.8636$ \\
   \rule{0pt}{15pt}$\hat{F}_2 = \begin{bmatrix}
   -1.0599  & -1.7036 &  -0.0013  & -0.0071 \\
   -0.1702 &  -0.0057  &  \phantom{-}0.0263  & -0.0582
   \end{bmatrix}$, $\lVert\hat{F}_2\rVert_F =  2.0146$ \\
     \rule{0pt}{15pt}$\hat{F}_3 = \begin{bmatrix}
    \phantom{-}3.3768  &  \phantom{-}2.2570  &  0.0410  & -0.0098\\
   -1.4694 &  -0.0061  &  0.1242  & -3.8379
   \end{bmatrix}$, $\lVert\hat{F}_3\rVert_F = 5.7795$ \\
    \rule{0pt}{10pt} Average \# of Newton descent iterations = $19.1$\\
   \hline		 \hline	\\[-0.7em]		
  \textbf{Sparse solutions by Algorithm \ref{algo:sp_des}} \\
  \hline 
   \rule{0pt}{21pt} $\hat{X}_1 =
  \begin{bmatrix} 
   -0.3370 + 0.3296i & \multirow{4}{*}{$\hat{x}_1^{*}$}  &  -0.4525 &  \phantom{-}0.5168 \\
   \phantom{-}0.2884 - 0.1698i  & &  \phantom{-}0.1764 &  -0.3007 \\
  -0.4705 + 0.5066i &  &  -0.7478 &   \phantom{-}0.7831 \\
   \phantom{-}0.2754 + 0.3347i  & &  -0.4527 &   \phantom{-}0.1711 
   \end{bmatrix}$ \\ 
    \rule{0pt}{21pt} $\hat{L}_1 =\begin{bmatrix}
  \phantom{-}25.2072 +16.5227i  & \multirow{4}{*}{$\hat{l}_1^{*}$} & \phantom{-}44.6547  & \phantom{-}82.3616 \\
  \phantom{-}11.0347 +11.8982i  & &  \phantom{-}12.5245 &  \phantom{-}35.8266 \\
 -12.2600 - 5.8549i &  & -24.5680 & -39.1736\\
  -0.6417 - 2.2060i  & &  -0.4383 &  -3.6464
   \end{bmatrix}$ \\
    \rule{0pt}{15pt}$\hat{F}_1 = \begin{bmatrix}
  0.9627 &  -1.3744  &  \phantom{-}0.0000  & \phantom{-}0.0000 \\
    0.4409 &  \phantom{-}0.0000  & -0.6774 &  -0.1599
   \end{bmatrix}$, $\lVert\hat{F}_1\rVert_F = 1.8694$ \\
   \rule{0pt}{15pt}$\hat{F}_2 = \begin{bmatrix}
    -1.0797  & -1.7362 &  0.0000  & \phantom{-}0.0000 \\
   -0.1677 &  \phantom{-}0.0000  &  0.0264  &  -0.0610
   \end{bmatrix}$, $\lVert\hat{F}_2\rVert_F = 2.0525$ \\
     \rule{0pt}{15pt}$\hat{F}_3 = \begin{bmatrix}
    \phantom{-}3.4465  &  \phantom{-}2.2568  &  0.0000  & \phantom{-}0.0000\\
   -1.8506  & -0.0000 &   0.3207 &  -4.0679
   \end{bmatrix}$, $\lVert\hat{F}_3\rVert_F = 6.0866$ \\
    \rule{0pt}{10pt} Average \# of Newton descent iterations:\\
    1. Using random initialization = $8231$\\
    2. Using initialization by Algorithm \ref{algo:unsp_des_pen} = $715$\\
   \hline
\end{tabular}
\end{center}
\label{tab:NS_AS_S_Soln}
\end{table} 

The eigenvalues of $A$ are $\Gamma(A) = \{-2,-1,1\pm 2j\}$. Thus, the feedback $F$ is required to move two unstable eigenvalues into the stable region while keeping the other two stable eigenvalues fixed. Table \ref{tab:NS_AS_S_Soln} shows the non-sparse, approximately-sparse and sparse solutions obtained by Algorithms \ref{algo:unsp_des}, \ref{algo:unsp_des_pen} and \ref{algo:sp_des}, respectively, and Figure \ref{fig:iter_run} shows a sample iteration run of these algorithms. Since the number of iterations taken by the algorithms to converge depends on their starting points, we report the average number of iterations taken over $1000$ random starting points. Further, to obtain approximately-sparse solutions, we use the weighing matrix with $W_{ij}=w$ if $\F_{ij}=1$. 
All the algorithms obtain three local minima, among which the first is the global minimum. The second column in $\hat{X}$ and $\hat{L}$ is conjugate of the first column (c.f. Remark \ref{rem:conj_eigvec}). It can be verified that the non-sparse and sparse solutions satisfy the optimality conditions of Theorem \ref{thm:opt_feedback}. 

For the non-sparse solution, the number of iterations taken by
Algorithm \ref{algo:unsp_des} with steepest descent are considerably
larger than the Newton descent. This is because the steepest descent
converges very slowly near a local minimum. Therefore, we use Newton
descent steps in Algorithms \ref{algo:unsp_des} and
\ref{algo:unsp_des_pen}.  Next, observe that the entries at the
sparsity locations of the locally minimum feedbacks obtained by
Algorithm \ref{algo:unsp_des_pen} have small magnitude. Further, the
average number of Newton descent iterations for convergence and the
norm of the feedback obtained of Algorithm \ref{algo:unsp_des_pen} is
larger as compared to Algorithm \ref{algo:unsp_des}. This is because
the approximately-sparse optimization problem \eqref{eq:opt_cost_pen}
is effectively more restricted than its non-sparse
counterpart~\eqref{eq:opt_cost_unsp2}.

Finally, observe that the solutions of Algorithm \ref{algo:sp_des} are
sparse. Note that Algorithm \ref{algo:sp_des} involves the use of
projection Algorithm \ref{algo:sp_EAP2}, which in turn involves
running Algorithm \ref{algo:unsp_des} multiple times. Thus, for a
balanced comparison, we present the total number of Newton descent
iterations of Algorithm \ref{algo:unsp_des} involved in the execution
of Algorithm \ref{algo:sp_des}.\footnote{The number of outer
  iteration of Algorithm \ref{algo:sp_des} are considerably less, for
  instance, $20$ in Figure \ref{fig:iter_run}.} From Table
\ref{tab:NS_AS_S_Soln}, we can observe that Algorithm
\ref{algo:sp_des} involves considerably more Newton descent iterations
compared to Algorithms \ref{algo:unsp_des} and
\ref{algo:unsp_des_pen}, since it involves computationally intensive
projection calculations by Algorithm \ref{algo:sp_EAP2}. One way to
reduce its computation time is to initialize is \emph{near} the local
minimum using the approximately sparse solution of Algorithm
\ref{algo:unsp_des_pen}.


\begin{figure}[h!]
\centering
\includegraphics[width=\columnwidth]{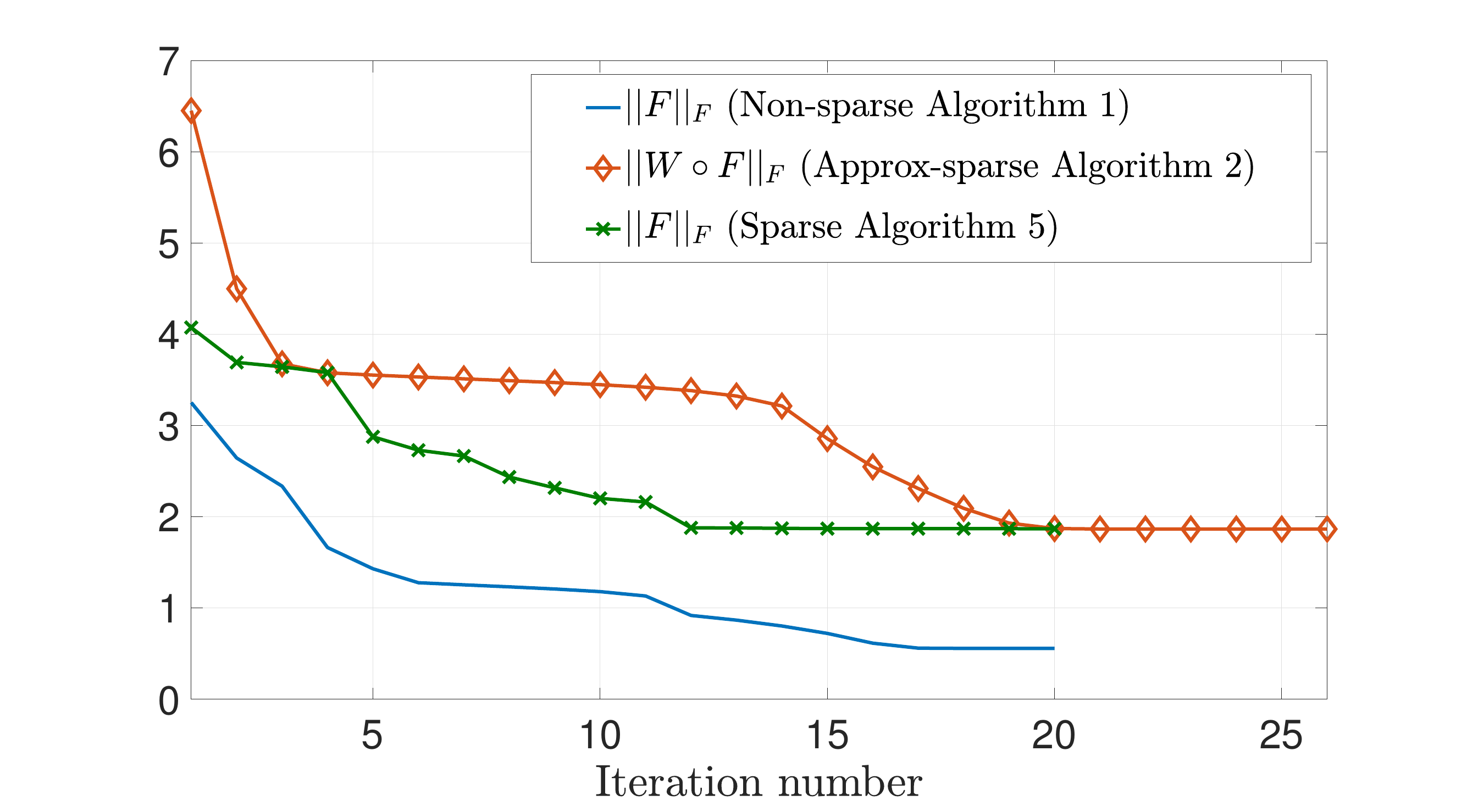}
\caption{Optimization costs for a sample run of Algorithms
  \ref{algo:unsp_des}, \ref{algo:unsp_des_pen} and \ref{algo:sp_des}
  (the algorithms converge to their global minima). For Algorithm
  \ref{algo:sp_des}, number of outer iterations are reported.}
\label{fig:iter_run}
\end{figure}

\begin{figure}[h!]
  \centering
  \subfigure[]{
  \includegraphics[width=.45\columnwidth]{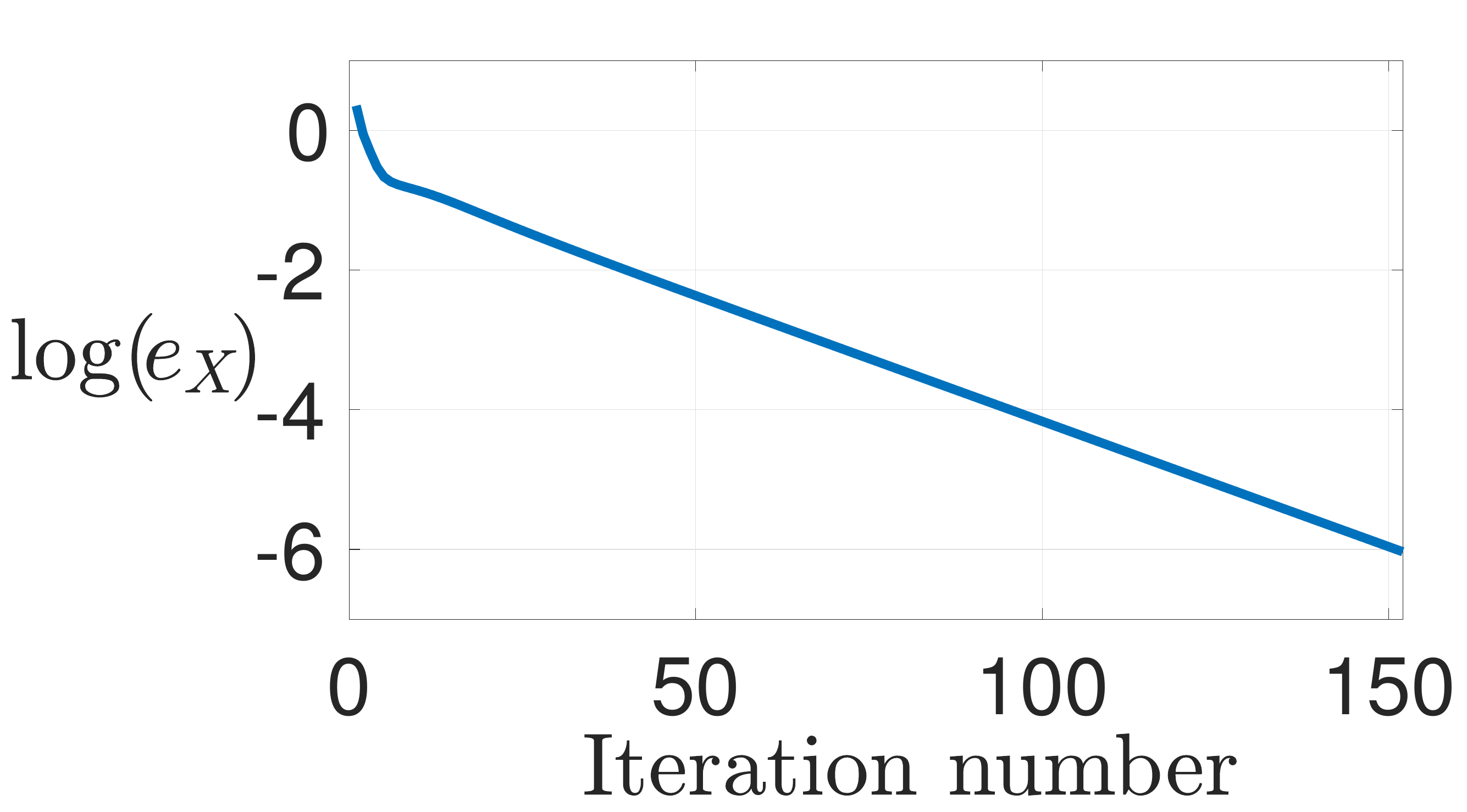} \label{fig:EAP1}} 
  \subfigure[]{
  \includegraphics[width=.45\columnwidth]{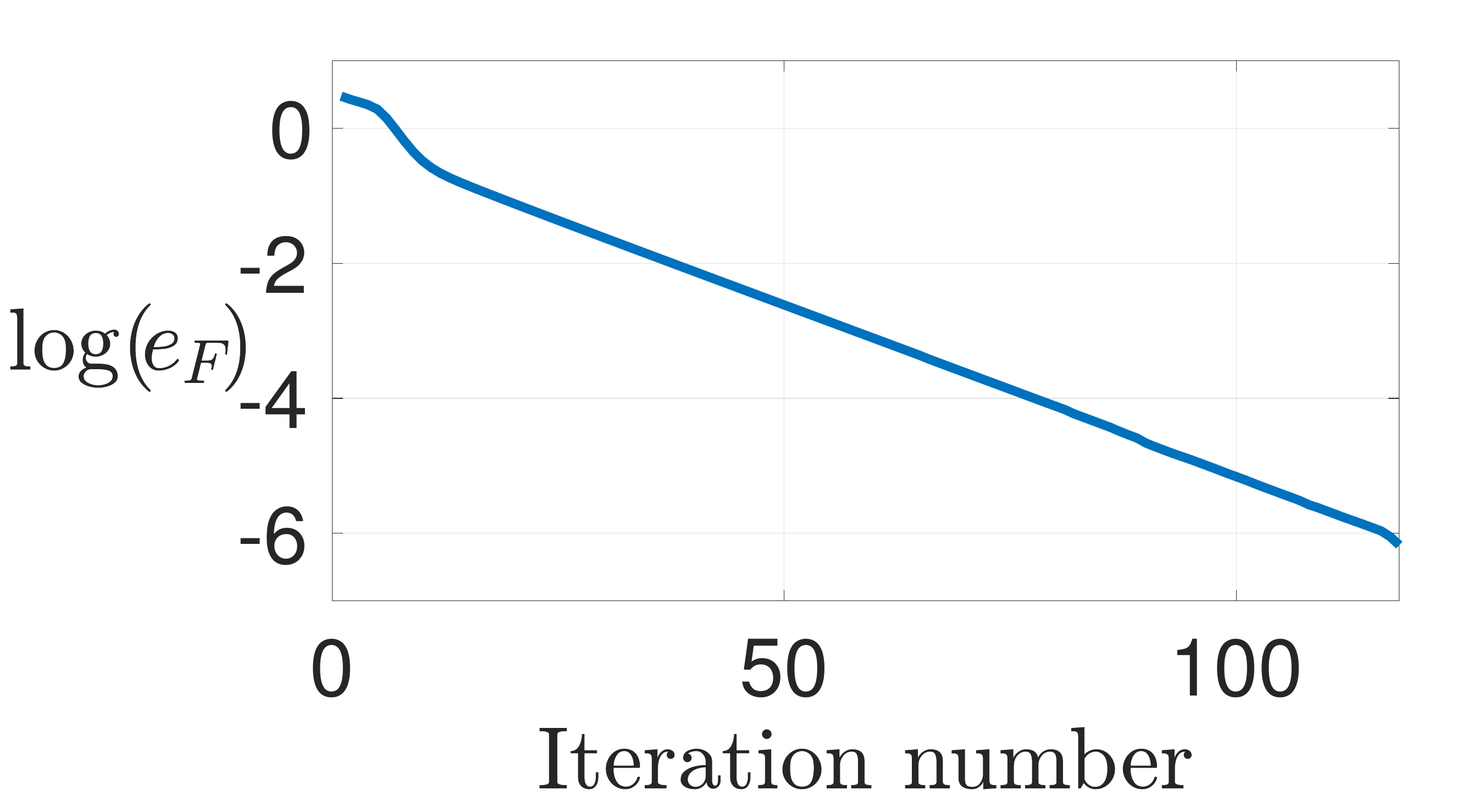} \label{fig:EAP2}}
  \caption{Projection and sparsity errors for a sample run of (a) Algorithm \ref{algo:sp_EAP1} and (b) Algorithm \ref{algo:sp_EAP2}, respectively.}
  \label{fig:EAP_sample_run}
\end{figure}

Figure \ref{fig:EAP_sample_run} shows a sample run of EAP Algorithms \ref{algo:sp_EAP1} and \ref{algo:sp_EAP2} for $G_0 = \left[\begin{smallmatrix} -1.0138 &   \phantom{-}0.6851 &  -0.1163 &   0.8929\\ -1.8230 &  -2.2041  & -0.1600  &  0.7293 \end{smallmatrix}\right]$. The sparse feedback obtained by Algorithms \ref{algo:sp_EAP1} and \ref{algo:sp_EAP2} are $F = \left[\begin{smallmatrix}\phantom{-}0.1528 &  -2.6710   &      0.0000     &    \phantom{-}0.0000 \\
    -0.8382   &      \phantom{-}0.0000  &  0.1775 &  -0.1768\end{smallmatrix}\right]$ and $F = \left[\begin{smallmatrix}  \phantom{-}4.2595  &  4.2938  & \phantom{-}0.0000  & \phantom{-}0.0000\\
    -0.2519 & 0.0000 & -2.1258 & -1.3991\end{smallmatrix}\right]$,
respectively. The projection error $e_X$ and the sparsity error $e_F$
capture the convergence of Algorithms \ref{algo:sp_EAP1} and
\ref{algo:sp_EAP2}, respectively. Figure \ref{fig:EAP_sample_run}
shows that these errors decrease, thus indicating convergence of the
algorithms.

Next, we provide an empirical verification of the convergence of
heuristic Algorithms \ref{algo:sp_EAP1} and \ref{algo:sp_EAP2}. Let
the sparsity ratio ($SR$) be defined as the ratio of the number of
sparse entries to the total number of entries in $F$ (i.e.
$SR = \frac{\text{Number of } 0's \text{ in } \F}{mn}$). We perform
$1000$ random executions of both the algorithms. In each execution,
$n$ is randomly selected between $4$ and $20$ and $m$ is randomly
selected between $2$ and $n$. Then, matrices $(A,B)$ are randomly
generated with appropriate dimensions. Next, a binary sparsity pattern
matrix $\F$ is randomly generated with the number of sparsity entries
given by $\lfloor SR\times mn \rfloor$, where $\lfloor \cdot \rfloor$
denotes rounding to the next lowest integer. To ensure feasibility
(c.f. Assumption \ref{assump:feasibility} and discussion below), we pick the
desired eigenvalue set $\mc{S}$ randomly as follows: we select a
random $F_r$ which satisfies the selected sparsity pattern $\F$, and
select $\mc{S} = \Gamma(A+BF_r)$. Finally, we set $iter_{max} = 1000$
and select a random starting point $G_0 (F_0,X_0)$, and run both
algorithms from the same starting point. Let $F_{sol}$ denote the
feedback solution obtained by Algorithms \ref{algo:sp_EAP1} and
\ref{algo:sp_EAP2}, respectively, and let
$d_{F_{sol},F_0} \triangleq \lVert F_{sol}-F_0 \rVert_F$ denote the
distance between the starting point $F_0$ and the final
solution. Since Algorithm \ref{algo:sp_EAP2} is a projection
algorithm, the metric $d_{F_{sol},F_0}$ captures its projection
performance.

\begin{table}[h!]
\caption{Convergence properties of EAP Algorithms \ref{algo:sp_EAP1} and \ref{algo:sp_EAP2}}
\begin{center}
 \begin{tabular}{| l | l | l | }
  \hline 
  $\mathbf{SR}$ & \multicolumn{1}{|c|}{\textbf{Algorithm \ref{algo:sp_EAP1}}}  & \multicolumn{1}{|c|}{\textbf{Algorithm \ref{algo:sp_EAP2}}}\\
  \hline
  \multirow{2}{*}{$1/4$}  & Convergence instances = 427  & Convergence instances = 997\\
   & Average $d_{F_{sol},F_0}$ = 8.47   & Average $d_{F_{sol},F_0}$ = 1.28 \\
  \hline
  \multirow{2}{*}{$1/2$}  & Convergence instances = 220  & Convergence instances = 988\\
   & Average $d_{F_{sol},F_0}$ = 4.91  & Average $d_{F_{sol},F_0}$ =    1.61 \\
  \hline
  \multirow{2}{*}{$2/3$}  & Convergence instances = 53  & Convergence instances = 983\\
   & Average $d_{F_{sol},F_0}$ = 6.12 & Average $d_{F_{sol},F_0}$ = 1.92 \\
  \hline
 \end{tabular}
\end{center}
\label{tab:EAP1_2_Soln}
\end{table} 

Table \ref{tab:EAP1_2_Soln} shows the convergence results of
Algorithms \ref{algo:sp_EAP1} and \ref{algo:sp_EAP2} for three
different sparsity ratios. While the convergence of Algorithm
\ref{algo:sp_EAP1} deteriorates as $F$ becomes more sparse, Algorithm
\ref{algo:sp_EAP2} converges in almost all instances. This implies
that in the context of Algorithm \ref{algo:sp_des}, Algorithm
\ref{algo:sp_EAP2} provides a valid projection in step
\ref{line:sp_MGEAP1} in almost all instances. We remark that in the
rare case that Algorithm \ref{algo:sp_EAP2} fails to converge, we can
reduce the step size $\alpha$ to obtain a new $G_{ns}$ and compute its
projection. Further, we compute the average of distance
$d_{F_{sol},F_0}$ over all executions of the Algorithms
\ref{algo:sp_EAP1} and \ref{algo:sp_EAP2} that converge. Observe that
the average distance for Algorithm \ref{algo:sp_EAP2} is smaller than
Algorithm \ref{algo:sp_EAP1}. This shows that Algorithm
\ref{algo:sp_EAP2} provides considerably better projection of $F_0$ in
the space of sparse matrices as compared to Algorithm
\ref{algo:sp_EAP1} (c.f. Remark \ref{rem:good_proj}). Note that the
above simulations focus on the convergence properties as $SR$ changes
and they do not capture the individual effects of the number of
sparsity constraints, $m$ and $n$ (for instance, small systems versus
large systems).

\section{Feasibility of the sparse EAP}\label{sec: feasibility}
In this section we provide a discussion of the feasibility of the
sparse EAP, i.e., eigenvalue assignment with sparse, static state
feedback. We address certain aspects of this problem, and leave a
detailed characterization for future research.

\begin{lemma} \textbf{(NP-hardness)} Determining the feasibility of
  the sparse EAP is NP-hard.
\end{lemma}
\begin{proof}
  We prove the result using the NP-hardness property of the static
  output feedback pole placement (SOFPP) problem \cite{MF:04}, which
  is stated as follows: for a given $A\in\real^{n\times n}$,
  $B\in\real^{n\times m}$, and $C\in\real^{p\times n}$, determine if
  there exists a non-sparse $K\in\real^{m\times p}$ such that the
  eigenvalues of $A+BKC$ are at some desired locations. Without loss
  of generality, we assume that $C$ is full row rank. Thus, there
  exists an invertible
  $T=\begin{bmatrix}C^{+} & V\end{bmatrix} \in\real^{n\times n}$ with
  $CV=0$. Taking the similarity transformation by $T$ (which preserves
  the eigenvalues) we get
\begin{align*}
T^{-1} (A+BKC) T = \underbrace{T^{-1}AT}_{\bar{A}} + \underbrace{T^{-1}B}_{\bar{B}}\begin{bmatrix} K & 0 \end{bmatrix}.
\end{align*} 
Clearly, the above SOFPP problem is equivalent to a sparse EAP with matrices $\bar{A},\bar{B}$ and sparsity pattern given by $\bar{F} = \begin{bmatrix} 1_{m\times p} & 0_{m\times (n-p)}\end{bmatrix}$. Thus, the NP-hardness of the sparse EAP follows form the NP-hardness of the SOFPP problem.
\end{proof}

Next, we present graph-theoretic necessary and sufficient conditions
for arbitrary eigenvalue assignment by sparse static feedback. Due to
space constraints, we briefly introduce the required graph-theoretic
notions and refer the reader to \cite{KJR:88} for more details. Given
a square matrix $A=[a_{ij}]\in\real^{n\times n}$, let $\mc{G}_A$
denote its associated graph with $n$ vertices, and let $a_{ij}$ be the
weight of the edge from vertex $j$ to vertex $i$. A closed directed
path (sequence of consecutive vertices) is called a cycle if the start
and end vertices coincide, and no vertex appears more than once along
the path (except for the first vertex). A set of vertex disjoint
cycles is called a cycle family. The width of a cycle family is the
number of edges contained in all its cycles.

\begin{lemma} \textbf{(Necessary
    conditions)} \label{lem:necc_cond_feasibility} Let
  $H =\left[ \begin{smallmatrix} A & B \\ F &
      0\end{smallmatrix}\right]$, and let $\mc{G}_H$ be its associated
  graph. Let $n_s$ be the number of sparsity constraints (zero
  entries) in the feedback matrix $F\in \real^{m\times n}$. Further,
  let $S_k$ denote the set of cycle families of $\mc{G}_H$ of width
  $k$, with $k=1,\dots,n$. Necessary conditions
  for arbitrary eigenvalue assignment with sparse, static state
  feedback are:
  \begin{enumerate}
  \item $n_s \le (m-1)n$, that is, $F$ has at least $n$ nonzero
    entries,
    
  \item for each $k=1,\cdots,n$, there exist a nonzero entry
    $f_{i_k j_k}$ of $F$ such that its corresponding edge appears in
    $S_k$.
  \end{enumerate}
\end{lemma}
\begin{proof} 
  (i) Arbitrary eigenvalue assignment requires that the feedback $F$ should
  assign all the $n$ coefficients of the characteristic polynomial
  $\text{det}(sI-A-BF)$ to arbitrary values. This requires the mapping
  $h:\real^{mn-n_s}\rightarrow \real^{n}$ from the nonzero entries
  of $F$ to the coefficients of the characteristic polynomial to be
  surjective, which imposes that the dimension of the domain of $h$
  should not be less that the dimension of its codomain.

  (ii) Let $c_k$, $k = 1,\dots,n$ denote the coefficients of the
  polynomial det($sI-A-BF$). Then, $c_k$ is a multiaffine function of
  the nonzero entries of $F$, which appear in the cycle families in
  $S_k$ \cite{KJR:88}. If there exists no feedback edge in $S_k$, then
  $c_k$ is fixed and does not depend on $F$. Thus, arbitrary
  eigenvalue assignment is not possible in this case.
\end{proof}

\begin{lemma}\textbf{(Sufficient
    conditions)}\label{lem:suff_cond_feasibility}
  Let
  $H = \left[ \begin{smallmatrix} A & B \\ F &
      0\end{smallmatrix}\right]$, and let $\mc{G}_H$ be its associated
  graph. Let $S_k$ denote the set of cycle families of $\mc{G}_H$ of
  width $k$, and let $F_k$ denote the set of feedback
  edges\footnote{Feedback edges are those associated with the nonzero
    entries of $F$.} contained in $S_k$, with $k=1,\dots,n$. Then, each
  of the following conditions is sufficient for arbitrary eigenvalue
  assignment with sparse, static state feedback:
  \begin{enumerate}
  \item for each $k=1,\cdots,n$, there exist a feedback edge that
    appears in $S_k$ and not in $S_j$, for all $j \neq k$,

    
  \item there exists a permutation $\{i_1,i_2,\cdots,i_n\}$ of
    $\{1,2,\cdots,n\}$ such that
    $\emptyset \neq F_{i_1}\subset F_{i_2} \subset \cdots \subset
    F_{i_n}$.
  \end{enumerate}
\end{lemma}
\begin{proof}
  (i) Similar to the proof of Lemma \ref{lem:necc_cond_feasibility},
  if there exist a nonzero entry $f_{i_k j_k}$ that is exclusive to $S_k$,
  then such variable can be used to assign $c_k$ arbitrarily. If this
  holds for $k=1,\dots,n$, then all coefficients of the characteristic
  polynomial can be
  assigned arbitrarily, resulting in arbitrary eigenvalue assignment.

  (ii) Condition (ii) guarantees that, for $j=2,\cdots,n$, the
  coefficient $c_{i_j}$ depends on the feedback variables of
  $c_{i_{j-1}}$ and on some additional feedback variables. These
  additional variables can be used to assign the coefficient $c_{i_j}$
  arbitrarily.
\end{proof}

To illustrate the results, consider the following example:
\begin{align*}
A = \begin{bmatrix} a_{11} & a_{12} & 0 \\ 0 & 0 & 0\\ a_{31} & 0 & 0\end{bmatrix}, B = \begin{bmatrix} b_{11} &  0 \\ 0 & b_{22} \\ 0 & 0 \end{bmatrix}, F = \begin{bmatrix} f_{11} & 0 & 0 \\ f_{21} & 0 & f_{23}\end{bmatrix}.
\end{align*}
The corresponding graph and cycle families are shown in
Fig.~\ref{fig:graph_illus}. Note that the edges $f_{11}, f_{21}$ and
$f_{23}$ are exclusive to cycle families of widths $1,2$ and $3$,
respectively. Thus, condition (i) of Lemma
\ref{lem:suff_cond_feasibility} is satisfied and arbitrary eigenvalue
assignment is possible. Next, consider the feedback
$F=\left[\begin{smallmatrix}0 & 0 & f_{13} \\ f_{21} & f_{22} &
    0 \end{smallmatrix}\right]$. In this case, the sets of feedback
edges in the family cycles of different widths are
$F_1 = \{f_{22}\}, F_2 = \{f_{22}, f_{13}, f_{21}\}$ and
$F_3 = \{f_{22},f_{13}\}$. We observe that
$F_1\subset F_3 \subset F_2$ (condition (ii) of Lemma
\ref{lem:suff_cond_feasibility}) and arbitrary eigenvalue assignment
is possible. Further, if
$F=\left[\begin{smallmatrix}0 & f_{12} & f_{13} \\ f_{21} & 0 &
    0 \end{smallmatrix}\right]$, then $F_1=F_3=\emptyset$, and the
coefficients $c_1$, $c_3$ of $\text{det}(sI-A-BF)$ are fixed. This
violates condition (ii) of Lemma \ref{lem:necc_cond_feasibility} and
prevents arbitrary eigenvalue assignment with the given $F$.

Note that the conditions in Lemmas
\ref{lem:necc_cond_feasibility} and \ref{lem:suff_cond_feasibility}
are constructive, and can also be used to determine a sparsity pattern
that guarantees feasibility of the sparse EAP. We leave the design of
such algorithm as a topic of future investigation.

\begin{remark} {\bf \emph{(Comparison with exiting results)}} We
  emphasize that the conditions presented in Lemmas
  \ref{lem:necc_cond_feasibility} and \ref{lem:suff_cond_feasibility}
  for arbitrary eigenvalue assignment using \emph{static} state
  feedback are not equivalent to the conditions for non-existence of
  SFMs studied in \cite{SHW-EJD:73}, \cite{MES-DDS:81,VP-MES-DDS:84,
    VP-MES-DDS:83, AA-MR:14a, AA-MR:14b}, \cite{SM-PC-MNB:18}. The
  reason is that arbitrary assignment of non-SFMs necessarily requires
  a \emph{dynamic} controller and cannot, in general, be achieved by a
  \emph{static} controller. Further, our graph-theoretic results are
  based on the feedback edges being suitably covered by cycle
  families, whereas the results in \cite{VP-MES-DDS:84, VP-MES-DDS:83}
  are based on the state nodes/subgraphs being suitably covered by
  strong components, cycles/cactus. \oprocend
\end{remark}

\begin{figure}[t]
  \centering \subfigure[Graph $\mc{G}_H$, where blue and orange nodes
  correspond to state and control vertices, respectively.]{
    \hspace{1pt} \includegraphics[scale=0.25,page=1]{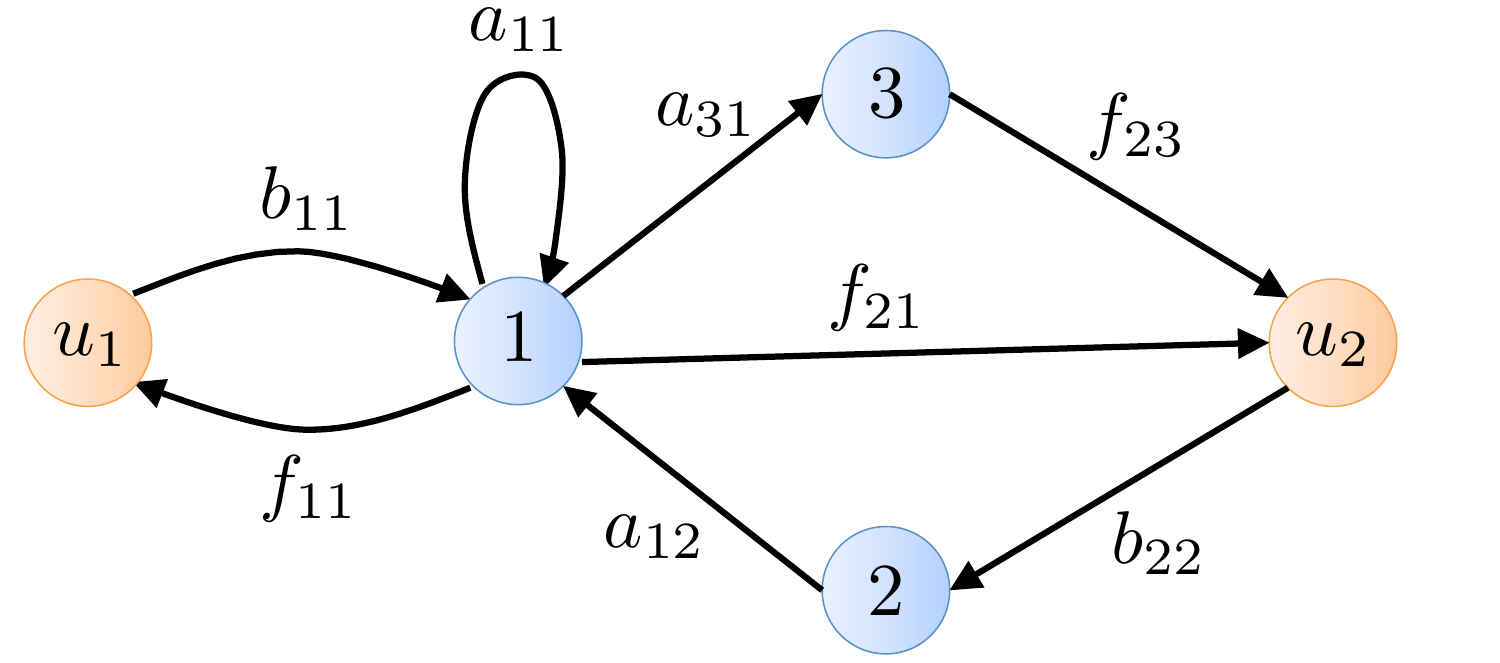} \label{fig:graph}}
  \subfigure[Cycle families of $\mc{G}_H$. All cycle families contain
  a single cycle. There are two cycle families of width $1$, and one
  cycle family each of width $2$ and $3$.]{
 \includegraphics[scale=0.25,page=2]{./img/graph} \label{fig:cycle_families}}
  \caption{Graph $\mc{G}_H$ and its cycle families. $u_1$, $u_2$
    denote the control vertices.}
  \label{fig:graph_illus}
\end{figure}

\section{Conclusion}\label{sec:conclusion}
In this paper we studied the MGEAP for LTI systems with arbitrary
sparsity constraints on the static feedback matrix. We presented an
analytical characterization of its locally optimal solutions, thereby
providing explicit relations between an optimal solution and the
eigenvector matrices of the associated closed loop system. We also
provided a geometric interpretation of an optimal solution of the
non-sparse MGEAP. Using a Sylvester-based parametrization, we
developed a heuristic projected gradient descent algorithm to obtain
local solutions to the MGEAP. We also presented two novel algorithms
for solving the sparse EAP and an algorithm to obtain approximately
sparse local solution to the MGEAP. Numerical studies suggest that our
heuristic algorithm converges in most cases. Further, we also
discussed the feasibility of the sparse EAP and provided necessary and
sufficient conditions for the same.


The analysis in the paper is developed, for the most part, under the
assumption that the sparse EAP problem with static feedback is
feasible. A future direction of research includes a more detailed
characterization of the feasibility of the EAP, a constructive
algorithm to determine feasible sparsity patterns, a convex relaxation
of the sparse MGEAP with guaranteed distance from optimality, and a
more rigorous analysis of convergence of Algorithms \ref{algo:sp_EAP2}
and \ref{algo:sp_des}.



\bibliographystyle{unsrt}
\bibliography{./bib/alias,./bib/Main,./bib/New,./bib/FP}

\end{document}